\documentclass[10pt]{article}

\usepackage{amssymb}   
\usepackage{amsthm}    
\usepackage{amsmath}   
\usepackage{stmaryrd}  
\usepackage{titletoc}  
\usepackage{mathrsfs}  
\usepackage{graphicx}
\usepackage{xcolor}
\usepackage[toc,page,title,titletoc,header]{appendix}
\usepackage{ulem}

\newlength{\defbaselineskip}
\newcommand{\setlinespacing}[1]%
           {\setlength{\baselineskip}{#1 \defbaselineskip}}

\newtheorem*{assumption}{Standing Assumption}

\newcommand{\bR}{\mathbb{R}}

\newcommand{\bZ}{\mathbb{Z}}

\makeatletter\@addtoreset{equation}{section} \makeatother
 \allowdisplaybreaks

\newtheorem{theorem}{Theorem}[section]
\newtheorem{lemma}[theorem]{Lemma}

\theoremstyle{definition}
\newtheorem{definition}[theorem]{Definition}

\newtheorem{proposition}[theorem]{Proposition}
\newtheorem{corollary}[theorem]{Corollary}

\theoremstyle{remark}
\newtheorem{remark}[theorem]{Remark}

\numberwithin{equation}{section}

\begin{document}

\title{Maximum Principle for Quasi-linear Reflected Backward SPDEs}

\author{ Guanxing Fu$^1$ \and Ulrich Horst$^1$ \and Jinniao Qiu$^2$}


\footnotetext[1]{Department of Mathematics, Humboldt-Universit\"{a}t zu Berlin, Unter den Linden 6, 10099 Berlin, Germany.
 Financial support from the Berlin Mathematical School and the SFB 649 ``Economic Risk'' is gratefully acknowledged.
  \textit{E-mail}: fuguanxing725@gmail.com (Guanxing Fu); horst@math.hu-berlin.de (Ulrich Horst).}

  \footnotetext[2]{  Department of Mathematics, University of Michigan, East Hall, 530 Church Street, Ann Arbor, MI 48109-1043, USA.
  \textit{Email}: qiujinn@gmail.com (Jinniao Qiu).}





\maketitle

\begin{abstract}
This paper establishes a maximum principle for quasi-linear reflected backward stochastic partial differential equations (RBSPDEs for short). We prove the existence and uniqueness of the weak solution to RBSPDEs allowing for non-zero Dirichlet boundary conditions and, using a stochastic version of De Giorgi's iteration, establish the maximum principle for RBSPDEs on a general domain. The maximum principle for RBSPDEs on a bounded domain and the maximum principle for backward stochastic partial differential equations (BSPDEs for short) on a general domain can be obtained as byproducts. Finally, the local behavior of the weak solutions is considered.
\end{abstract}



\section{Introduction}
Let $(\Omega,\mathcal{F},\mathbb{P})$ be a complete probability space carrying a standard $m$-dimensional Brownian motion $W=\{W_t,t\geq 0\}$. Let $(\mathcal{F}_t)_{t \geq 0}$ be the natural filtration generated by $W,$ augmented by the $\mathbb{P}$-null sets in $\mathcal{F}$. In this paper, we establish a maximum principle for weak solutions to the reflected backward stochastic partial differential equation (RBSPDE)

\begin{equation}\label{equation1.1}
  \left\{\begin{array}{l}
  \begin{aligned}
-du(t,x)&=[\partial _j(a^{ij}\partial _iu(t,x)+\sigma ^{jr}v^r(t,x))
+f(t,x,u(t,x),\nabla u(t,x),v(t,x))\\
&\quad+\nabla \cdot g(t,x,u(t,x),\nabla u(t,x),v(t,x))]\,dt+\mu (dt,x)-v^r(t,x)dW^r_t,\\
&\quad~~~(t,x)\in Q:=[0,T]\times \mathcal{O},\\
u(T,x)&=G(x),~~~x\in \mathcal{O},\\
u(t,x)&\geq \xi (t,x) ~dt\times dx\times d\mathbb{P}-a.e.,\\
\int_Q(u(t,x)&-\xi (t,x))\mu (dt,dx)=0,
    \end{aligned}
  \end{array}\right.
\end{equation}
with general Dirichlet boundary conditions. Here and in what follows, the usual summation convention is applied, $\xi$ is a given stochastic process defined on $(\Omega,\mathcal{F},(\mathcal{F}_t)_{t\geq 0},\mathbb{P})$, called the {\sl obstacle process}, $T \in (0,\infty)$ is a deterministic {\sl terminal time}, $\mathcal{O}\subset \mathbb{R}^n$ is a possibly unbounded domain,  $\partial_ju=\frac{\partial u}{\partial x_j}$ and $\nabla=(\partial_{1},\cdots,\partial_{d})$ denotes the gradient operator. A {\sl solution} to the RBSPDE is a random triple $(u,v,\mu)$ defined on $\Omega\times[0,T]\times\bR^n$ such that (\ref{equation1.1}) holds in a suitable sense.

Since the introduction by Bensoussan \cite{Bensousan_83} backward stochastic partial differential equations (BSPDEs) have been extensively investigated in the probability and stochastic control literature. They naturally arise in many applications, for instance as stochastic Hamilton-Jacobi-Bellman equations associated with non-Markovian control problems \cite{peng}, as adjoint equations of the Duncan-Mortensen-Zakai equaltion in nonlinear filtering \cite{xhou} and as adjoint equations in stochastic control problems when formulating stochastic maximum principles \cite{Bensousan_83}. BSPDEs with singular terminal conditions arise in non-Markovian models for financial mathematics to describe optimal trading in illiquid financial markets \cite{paulwinqiuhorst}.

Reflected BSPDEs arise as the Hamilton-Jacobi-Bellman equation for the optimal stopping problem of stochastic differential equations with random coefficients \cite{changpangyong,qiuwei}, and as the adjoint equations for the maximum principle of Pontryagin type in singular control problems of stochastic partial differential equations in, e.g. \cite{oksendalsulemzhang}

Existence and uniqueness of solutions results for reflected PDEs and SPDEs have been established by many authors.
Pierre \cite{pierre,pierre1980} has studied parabolic PDEs with obstacles using parabolic potentials. Using methods and techniques from parabolic potential theory Denis, Matoussi and Zhang \cite{denismzhang} proved existence and uniqueness of solutions results for quasi-linear SPDEs driven by infinite dimensional Brownian motion. More recently, Qiu and Wei \cite{qiuwei} established a general theory of existence and uniqueness of solution for a class of quasi-linear RBSPDEs, which includes the classical results on obstacle problems for deterministic parabolic PDEs as special cases.

Adapting Moser's iteration scheme to the nonlinear case Aronson and Serrin \cite{AronsonSerrin1967} proved the maximum principle and local bounds of weak solutions for deterministic quasi-linear parabolic equations on bounded domains. Their method was extended by Denis, Matoussi, and Stoica \cite{DenisMatoussiStoica2005} to the stochastic case, obtaining an $L^p$ a priori estimate for the uniform norm of the solution of the stochastic quasi-linear parabolic equation with null Dirichlet condition, and further adapted by Denis, Matoussi, and Stoica \cite{denis-matoussi-stoica-2009} to local solutions. Later, Denis, Matoussi, and Zhang \cite{denismatoussizhang} established $L^p$ estimates for the uniform norm in time and space of weak solutions to reflected quasi-linear SPDEs along with a maximum principle for local solutions using a stochastic version of Moser's iteration scheme. Recently, Qiu and Tang \cite{qiutang} used the De Giorgi's iteration scheme, a technique that also works for degenerate parabolic equations, to establish a local and global maximum principle for weak solutions of BSPDEs without reflection. To the best of our knowledge a maximum principle for reflected BSPDEs has not yet been established in the literature.

In this paper we establish a maximum principle for reflected BSPDEs on possibly unbounded domains; a maximum principle and a comparison principle for BSPDEs on general domains, a maximum principle for RBSPDEs on bounded domains and a local maximum principle for RBSPDEs are obtained as well. Due to the obstacle, the maximum principle for RBSPDE is not a direct extension of that for BSPDE in \cite{qiutang}. Our proofs rely on a stochastic version of De Giorgi's iteration scheme that does not depend on the Lebesgue measure of the domain; this extends the scheme in \cite{qiutang} that only applies to bounded domains. Our iteration scheme requires an almost sure representation of the $L^2$ norm of the positive part of the of the weak solution of RBSPDEs. This, in turn requires generalizing the It\^o's formula for weak solutions to BSPDEs established in \cite{qiutang} and \cite{qiuwei} to the positive part of weak solutions.

It is worth pointing out that by contrast to $L^p$ estimates ($p\in (2,\infty)$) for the time and space maximal norm of weak solutions to \textit{forward} SPDEs or related obstacle problems as established in \cite{DenisMatoussiStoica2005,denismatoussizhang,qiu-2015-DSPDE}, our estimate for weak  solutions is uniform with respect to $w\in \Omega$ and hence establishes an $L^{\infty}$ estimate. This distinction comes from the essential difference between BSPDEs and \textit{forward} SPDEs: the noise term in the former endogenously originates from martingale representation and is hence governed by the coefficients, while the latter is fully exogenous, which prevents any $L^\infty$ estimate  for \textit{forward} SPDEs.

The paper is organized as follows: in Section 2, we list some notations and the standing assumptions on the parameters of the RBSPDE \eqref{equation1.1}. The existence and uniqueness of weak solution to the RBSPDE \eqref{equation1.1} is presented in Section 3. In Sections 4, we establish the maximum principle for the RBSPDE (\ref{equation1.1}) on a general domain as well as the maximum principles for RBSPDEs on a bounded domain and BSPDEs on a general domain. The local behavior of the weak solutions to (\ref{equation1.1}) is also considered. Finally, we list in the appendix some useful lemmas, the frequently used It\^o formulas and some definitions related to the stochastic regular measure.

\section{Preliminaries and standing assumptions}

%

 For an arbitrary domain $\Pi$ in some Euclidean space, let $\mathcal{C}^\infty_0(\Pi)$ be the class of infinitely differentiable functions with compact support in $\Pi$, and $L^2(\Pi)$ be the usual square integrable space on $\Pi$ with the scalar product $\langle u,v\rangle_{\Pi}=\int_{\Pi}u(x)v(x)dx$ and the norm $\|u\|_{L^2(\Pi)}=\langle u,u\rangle^{\frac{1}{2}}_{\Pi}$ for each pair $u,v\in L^2(\Pi)$. For $(k,p)\in\bZ\times [1,\infty)$ where $\bZ$ is the set of all the integers, let $H^{k,p}(\Pi)$ be the usual $k$-th order Sobolev space. For convenience, when $\Pi=\mathcal{O}$, we write $\langle\cdot,\cdot\rangle$ and $\|\cdot\|$ for $\langle \cdot,\cdot\rangle_{\mathcal{O}}$ and $\|\cdot\|_{L^2(\mathcal{O})}$ respectively. We recall that $Q = [0,T] \times {\cal O}$.

For $t\in[0,T]$ and $\Pi\subseteq\mathbb{R}^n$, we put $\Pi _t :=[t,T]\times \Pi$. Denote by $H^{k,p}_{\mathcal{F}}(\Pi _t)$ the class of $H^{k,p}(\Pi)$-valued predictable processes on $[t,T]$ such that for each $u\in H^{k,p}_{\mathcal{F}}(\Pi _t)$ we have that
\begin{equation*}
\|u\|_{H^{k,p}_{\mathcal{F}}(\Pi _t)}:=\left(E\left[\int_t^T\|u(s,\cdot)\|^p_{H^{k,p}(\Pi)}ds\right]\right)^{1/p}<\infty.
\end{equation*}
Let $\mathcal{M}^{k,p}(\Pi _t)$ be the subspace of $H^{k,p}_{\mathcal{F}}(\Pi _t)$ such that
\begin{equation*}
\|u\|_{k,p;\Pi _t}:=\left(\textrm{esssup}_{\omega\in \Omega}\sup_{s\in[t,T]}E\left[\int_s^T\|u(\omega ,\tau , \cdot)\|^p_{H^{k,p}_{\mathcal{F}}(\Pi)}d\tau | \mathcal{F}_s\right]\right)^{1/p}<\infty
\end{equation*}
and  $\mathcal{L}^{\infty}(\Pi _t)$ be the subspace of $H^{0,p}_{\mathcal{F}}(\Pi _t)$ such that
\begin{equation*}
\|u\|_{\infty;\Pi _t}:=\textrm{esssup}_{(\omega ,s,x)\in \Omega \times \Pi _t}|u(\omega ,s,x)|<\infty.
\end{equation*}
Denote by $\mathcal{L}^{\infty,p}(\Pi _t)$ the subspace of $H^{0,p}_{\mathcal{F}}(\Pi _t)$ such that
\begin{equation*}
\|u\|_{\infty,p;\Pi _t}:=\textrm{esssup}_{(\omega , s)\in \Omega \times [t,T]}\|u(\omega,s,\cdot)\|_{L_p(\Pi)}<\infty.
\end{equation*}
Let $\mathcal{V}_2(\Pi _t)$ be the class of all $u\in H^{1,2}_{\mathcal{F}}(\Pi_t)$ such that
\begin{equation*}
\|u\|_{\mathcal{V}_2(\Pi _t)}:=\left(\|u\|_{\infty,2;\Pi_t}^2+\|\nabla u\|_{0,2;\Pi _t}^2\right)^{1/2}<\infty
\end{equation*}
and let $\mathcal{V}_{2,0}(\Pi _t)$ be the subspace of $\mathcal{V}_2(\Pi _t)$ for which
\begin{equation*}
\lim\limits_{r\rightarrow 0}||u(s+r,\cdot)-u(s,\cdot)||_{L^2(\Pi)}=0 ~~~\textrm{for}~\textrm{all}~s,s+r\in[t,T],\quad \text{a.s.}
\end{equation*}



\begin{assumption}
We assume throughout that the coefficients and the obstacle process of the RBSPDE (1.1) satisfy the following conditions. Denote by $\mathbb{F}$ the $\sigma$-algebra generated by all predictable sets on $\Omega\times[0,T]$ associated with $(\mathcal{F}_t)_{t\geq 0}$.

\begin{itemize}
	\item [($\mathcal{A}_1$)] The random functions
\[
	g(\cdot,\cdot,\cdot,X,Y,Z):\Omega \times [0,T]\times \mathcal{O}\rightarrow \mathbb{R}^n \quad  \text{and}\quad f(\cdot,\cdot,	\cdot,X,Y,Z):\Omega \times [0,T]\times \mathcal{O}\rightarrow \mathbb{R}
\]
are $\mathbb{F}\otimes \mathcal{B}(\mathcal{O})$-measurable for any $(X,Y,Z)\in \mathbb{R}\times\mathbb{R}^n\times\mathbb{R}^m$ and there exist positive constants $L$, $\kappa$ and $\beta$ such that for each $(X_i,Y_i,Z_i)\in \mathbb{R}\times\mathbb{R}^n\times\mathbb{R}^m$, $i=1,2$,
$$|g(\cdot,\cdot,\cdot,X_1,Y_1,Z_1)-g(\cdot,\cdot,\cdot,X_2,Y_2,Z_2)|\leq L|X_1-X_2|+\frac{\kappa}{2}|Y_1-Y_2|+\sqrt{\beta}|Z_1-Z_2|$$
and
$$|f(\cdot,\cdot,\cdot,X_1,Y_1,Z_1)-f(\cdot,\cdot,\cdot,X_2,Y_2,Z_2)|\leq L(|X_1-X_2|+|Y_1-Y_2|+|Z_1-Z_2|).$$

	\item[($\mathcal{A}_2$)] The coefficients $a$ and $\sigma$ are $\mathbb{F}\otimes \mathcal{B}(\mathcal{O})$-measurable and there exist positive constants $\varrho >1$, $\lambda$ and $\Lambda$ such that for each $\eta \in \mathbb{R}^n$ and $(\omega,t,x)\in\Omega\times[0,T]\times\mathcal{O}$,
\begin{align*}
	\lambda |\eta|^2\leq (2a^{ij}(\omega ,t,x)-\varrho \sigma ^{ir}\sigma ^{jr}(\omega,t,x))\eta ^i\eta ^j &\leq \Lambda |\eta|^2\\
	|a(\omega ,t,x)|+|\sigma (\omega ,t,x)| &\leq \Lambda,
\end{align*}
and
$$\lambda -\kappa -\varrho '\beta >0 ~with~\varrho ':=\frac{\varrho}{\varrho -1}.$$

	\item[($\mathcal{A}_3$)]
The terminal value satisfies $G\in L^{\infty}(\Omega , \mathcal{F}_T, L^2(\mathcal{O}))\cap L^{\infty}(\Omega,\mathcal{O})$ and for some $p>\max\{n+2,2+4/n\}$, one has
\begin{align*}
	g_0 & :=g(\cdot,\cdot,\cdot,0,0,0)\in \mathcal{M}^{0,p}(Q)\cap \mathcal{M}^{0,2}(Q)\\
	f_0 & :=f(\cdot,\cdot,\cdot,0,0,0)\in \mathcal{M}^{0,\frac{p(n+2)}{p+n+2}}(Q)\cap\mathcal{M}^{0,2}(Q).
\end{align*}

	\item[($\mathcal{A}_4$)]	The obstacle process $\xi$ is almost surely quasi-continuous (see Appendix for the definition) on $Q$ and there exists a process $\hat{\xi}$ such that $\xi \leq \hat{\xi}$ $ds\times dx\times d\mathbb{P}$-a.e., where $\hat{\xi}\in \mathcal{V}_{2,0}(Q)$ together with some $\hat{v}\in\mathcal{M}^{0,2}(Q)$ is a solution to BSPDE
\begin{equation}\label{xi-hat}
\left\{\begin{array}{ll}
\begin{split}
-d\hat{\xi}(t,x)&=[\partial _j(a^{ij}\partial _i\hat{\xi}(t,x)+\sigma ^{jr}\hat{v}^r(t,x))
+\hat{f}(t,x)+\nabla \cdot \hat{g}(t,x)]dt
-\hat{v}^r(t,x)dW^r_t,\\
&~~~\quad (t,x)\in Q,\\
\hat{\xi}(T,x)&=\hat{G}(x),~~~x\in \mathcal{O},
\end{split}
\end{array}
\right.
\end{equation}
with the random functions $\hat{f}$, $\hat {g}$ and $\hat{G}$ satisfying
\begin{align*}
	 &\hat{G} \in L^{\infty}(\Omega , \mathcal{F}_T, L^2(\mathcal{O}))\cap L^{\infty}(\Omega,\mathcal{O}),
	 &\hat{f} \in \mathcal{M}^{0,\frac{p(n+2)}{p+n+2}}(Q)\cap\mathcal{M}^{0,2}(Q),
	 &~~\hat{g} \in \mathcal{M}^{0,p}(Q)\cap\mathcal{M}^{0,2}(Q).
\end{align*}

\item[($\mathcal{A}_5$)] The function $x \mapsto g(\cdot,\cdot,\cdot,x,0,0)$ is uniformly Lipschitz continuous in norm:
\begin{align*}
	\|g(\cdot,\cdot,\cdot,X_1,0,0)-g(\cdot,\cdot,\cdot,X_2,0,0)\|_{0,p;Q}\leq L|X_1-X_2|;\\
    \|g(\cdot,\cdot,\cdot,X_1,0,0)-g(\cdot,\cdot,\cdot,X_2,0,0)\|_{0,2;Q}\leq L|X_1-X_2|.
\end{align*}
\end{itemize}
\end{assumption}

\begin{remark}

While the assumptions $(\mathcal{A}_1-\mathcal{A}_4)$ are standard for the existence and uniqueness of solution, the assumption $\mathcal{A}_5$ is required for the iteration scheme for proof of the maximum principle in Theorem \ref{Theorem_MP_GD} below, which follows easily from $(\mathcal{A}_1)$ when the domain is bounded.

\end{remark}

For the index $p$ specified in $(\mathcal{A}_3)$ and $t\in[0,T]$, define the functional $A_p$ and $B_2$ as follows:
\begin{equation*}
A_p(l,h;\mathcal{O}_t):=\|l\|_{0,\frac{p(n+2)}{p+n+2};\mathcal{O}_t}+\|h\|_{0,p;\mathcal{O}_t}, ~~(l,h)\in \mathcal{M}^{0,\frac{p(n+2)}{p+n+2}}(\mathcal{O}_t)\times \mathcal{M}^{0,p}(\mathcal{O}_t)
\end{equation*}
and
\begin{equation*}
B_2(l,h;\mathcal{O}_t):=\|l\|_{0,2;\mathcal{O}_t}+\|h\|_{0,2;\mathcal{O}_t}, ~~(l,h)\in \mathcal{M}^{0,2}(\mathcal{O}_t)\times \mathcal{M}^{0,2}(\mathcal{O}_t).
\end{equation*}

In Sections 3 and 4, we will repeatedly use the Young inequality of the form
\begin{equation}\label{repeatedestimate}
\langle f,g\rangle=\langle \sqrt{\epsilon}f,\frac{1}{\sqrt{\epsilon}}g\rangle \leq \frac{1}{2}\left[\epsilon \|f\|^2+\frac{1}{\epsilon}\|g\|^2\right].
\end{equation}


\section{Existence and uniqueness of weak solution to RBSPDE \eqref{equation1.1}}
In this section we prove an existence and uniqueness of weak solutions result for the RBSPDE \eqref{equation1.1} along with a strong norm estimate. The difficulty in defining weak solutions to the RBSPDE \eqref{equation1.1} is the random measure $\mu$. It is typically a local time so the Skorokhod condition $\int_Q(u-\xi)\,\mu(dt,dx)=0$ might not make sense. To give a rigorous meaning to the integral condition, the theory of parabolic potential and capacity introduced by \cite{pierre,pierre1980} was generalized by \cite{qiuwei} to a backward stochastic framework. We recall the definition of quasi continuity and stochastic regular measures in Appendix \ref{appendix-measure}.

\begin{definition}\label{definition3.1}
The triple $(u,v,\mu)$ is called a weak solution to the RBSPDE \eqref{equation1.1} if:
\begin{itemize}
	\item[(1)] $(u,v)\in \mathcal{V}_{2,0}(Q)\times \mathcal{M}^{0,2}(Q)$ and $\mu$ is a stochastic regular measure;
	\item[(2)] the RBSPDE \eqref{equation1.1} holds in the weak sense, i.e., for each $\varphi \in \mathcal{C}_0^{\infty}(\mathbb{R}^+)\otimes \mathcal{C}_0^{\infty}(\mathcal{O})$, we have
\begin{align*}
& \langle u(t,\cdot),\varphi (t,\cdot)\rangle \\
= &
\langle G(\cdot),\varphi(T,\cdot)\rangle
 -\int_t^T\left\{\langle u(s,\cdot),\partial _s\varphi (s,\cdot)\rangle+\langle\partial _j\varphi (s,\cdot),a^{ij}(s,\cdot)\partial _iu(s,\cdot) +\sigma ^{jr}v^r(s,\cdot)\rangle\right\}ds \\
&+\int_t^T\left[\langle f(s,\cdot,u(s,\cdot),\nabla u(s,\cdot),v(s,\cdot)),\varphi (s,\cdot)\rangle
-\langle g^j(s,\cdot,u(s,\cdot),\nabla u(s,\cdot),v(s,\cdot)),\partial_j \varphi (s,\cdot)\rangle\right]ds\\
&+\int_{[t,T]\times\mathcal{O}}\varphi (s,x)\mu (ds,dx)-\int_t^T\langle\varphi (s,\cdot),v^r(s,\cdot)dW_s^r\rangle,\quad\text{a.s.};
\end{align*}
	\item[(3)] $u$ admits a quasi-continuous version $\tilde{u}$ such that $\tilde{u}\geq \xi~ds\times dx\times d\mathbb{P}~a.e.$ and
\begin{equation}
\int_Q(\tilde{u}(t,x)-\xi (t,x))\mu (dt,dx)=0~~\mathbb{P}\text{-a.s.}
\end{equation}
\end{itemize}
\end{definition}

We denote by $\mathcal{U}(\xi,f,g,G)$ the set of all the weak solutions of the RBSPDE (\ref{equation1.1}) associated with the obstacle process $\xi$, the terminal condition $G$, and the coefficients $f$ and $g$. Further, $\mathcal{U}(-\infty,f,g,G)$ is the set of solutions  when there is no obstacle, i.e., $\mathcal{U}(-\infty,f,g,G)$ is the set of solution pairs $(u,v)$ to the associated BSPDE with terminal condition $G$ and coefficients $f$ and $g$. 

The following theorem guarantees the existence and uniqueness of weak solutions in the sense of Definition \ref{definition3.1}. The arguments for the norm estimate also apply to Lemma \ref{estimate-u-k} below, which is needed for the proof of our maximum principle.

\begin{theorem}\label{theorem3.2}
Suppose that Assumptions $(\mathcal{A}_1)$-$(\mathcal{A}_4)$ hold and that $\hat{\xi}|_{\partial\mathcal{O}}=0$. Then the RBSPDE (1.1) admits a unique solution $(u,v,\mu)$ that satisfies the zero Dirichlet condition $u|_{\partial\mathcal{O}}=0$. Moreover, for each $t\in [0,T]$, one has
\begin{equation}\label{equation3.2}
\begin{split}
\|u\|_{\mathcal{V}_2(\mathcal{O}_t)}+\|v\|_{0,2;\mathcal{O}_t}\leq &
C\left( \textrm{esssup}_{\omega \in \Omega}\|G(\omega,\cdot)\|_{L^2(\mathcal{O})}+\textrm{esssup}_{\omega \in \Omega}\|\hat{G}(\omega,\cdot)\|_{L^2(\mathcal{O})}\right. \\
&\left.+B_2(f_0,g_0;\mathcal{O}_t)+B_2(\hat{f},\hat{g};\mathcal{O}_t)\right),\\
\end{split}
\end{equation}
where the positive constant $C$ only depends on the constants $\lambda$, $\varrho$, $\kappa$, $\beta$, $L$  and $T$.
\end{theorem}

\begin{proof}
It has been shown in \cite[Theorem 4.12]{qiuwei} that the RBSPDE (1.1) admits a unique solution $(u,v,\mu)$ satisfying the zero Dirichlet condition $u|_{\partial\mathcal{O}}=0$ and that this solutions satisfies the integrability condition
\begin{equation*}
E\left[\sup_{t\in[0.T]}\|u(t)\|^2\right]+E\left[\int_0^T\|\nabla u(t)\|^2\,dt\right]+E\left[\int_0^T\|v(t)\|^2\,dt\right]<\infty.
\end{equation*}
Hence, we only need to prove the estimate (\ref{equation3.2}). To this end, notice first that
\begin{align*}
&\int_t^T\int_{\mathcal{O}}(u(s,x)-\hat{\xi}(s,x))\,\mu(dsdx)\\
=& \int_t^T\int_{\mathcal{O}}(u(s,x)-\xi(s,x)+\xi(s,x)-\hat{\xi}(s,x))\,\mu(dsdx)\\
\leq& 0.
\end{align*}
Thus for each $t\in [0,T]$, Proposition \ref{proposition6.1} yields almost surely,
\begin{equation} \label{estimate_begin}
\begin{split}
&\|u(t)-\hat{\xi}(t)\|^2+\int_t^T\|v(s)-\hat{v}(s)\|^2ds
\\
=&\|G-\hat{G}\|^2-\int_t^T\langle u(s)-\hat{\xi}(s),v^r(s) -\hat{v}^r(s)\rangle\, dW_s^r \\
& - \int_t^T\langle 2\partial _j(u-\hat{\xi}(s)),a^{ij}\partial _i(u-\hat{\xi})(s)  +\sigma ^{jr}(v^r-\hat{v}^r)\rangle\, ds \\
& -\int_t^T\langle 2\partial _j(u-\hat{\xi}(s)),g^j(s,u(s),\nabla u(s),v(s))-\hat{g}^j(s)\rangle\, ds \\
& +\int_t^T\langle 2(u-\hat{\xi}(s)),f(s,u(s),\nabla u(s),v(s))-\hat{f}(s)\rangle\, ds \\
& +\int_{\mathcal{O}_t}2(u(s,x)-\hat{\xi}(s,x))\mu (ds,dx)\\
\leq & \|G-\hat{G}\|^2-\int_t^T\langle u(s)-\hat{\xi}(s),v^r(s)-\hat{v}^r(s)\rangle\, dW_s^r \\
& -\int_t^T\langle 2\partial _j(u-\hat{\xi}(s)),a^{ij}\partial _i(u-\hat{\xi})(s)+\sigma ^{jr}(v^r-\hat{v}^r)\rangle \,ds \\
&-\int_t^T\langle 2\partial _j(u-\hat{\xi}(s)),g^j(s,u(s),\nabla u(s),v(s))-\hat{g}^j(s)\rangle \,ds \\
& +\int_t^T\langle 2(u-\hat{\xi}(s)),f(s,u(s),\nabla u(s),v(s))-\hat{f}(s)\rangle \,ds.
\end{split}
\end{equation}
Applying assumption $(\mathcal{A}_2)$ and (\ref{repeatedestimate}), one has
\begin{align*}
I_1:&
=-E\left[\int_t^T\langle 2\partial _j(u-\hat{\xi}(s)),a^{ij}\partial _i(u-\hat{\xi})(s)+\sigma ^{jr}(v_r-\hat{v}^r)\rangle \,ds\Big|\mathcal{F}_t\right]\\
=& -E\left[\int_t^T\langle\partial _j(u-\hat{\xi}(s)),(2a^{ij}-\sigma^{ir}\sigma^{jr}\varrho)\partial _i(u-\hat{\xi})(s)+\sigma^{ir}\sigma^{jr}\varrho\partial _i(u-\hat{\xi})(s)+2\sigma ^{jr}(v^r-\hat{v}^r)\rangle\, ds|\mathcal{F}_t\right]\\
\leq & -\lambda E\left[\int_t^T\|\nabla (u(s)-\hat{\xi}(s))\|^2\,ds|\mathcal{F}_t\right]+\frac{1}{\varrho}E\left[\int_t^T\|v(s)-\hat{v}(s)\|^2\,ds|\mathcal{F}_t\right].
\end{align*}
By assumptions $(\mathcal{A}_1)$ and $(\mathcal{A}_3)$ and the estimate (\ref{repeatedestimate}) it holds for each $\epsilon>0$ and $\theta>0$ that:
\begin{align*}
I_2:&
= -E\left[\int_t^T\langle 2\partial _j(u-\hat{\xi}(s)),g^j(s,u(s),\nabla u(s),v(s))-\hat{g}^j(s)\rangle\, ds|\mathcal{F}_t\right]\\
   \leq & E\left[\int_t^T\langle 2|\nabla(u(s)-\hat{\xi}(s))|,L|u(s)|+\frac{\kappa}{2}|\nabla u(s)|+\sqrt{\beta}|v(s)|\rangle\, ds|\mathcal{F}_t\right]\\
   &+E\left[\int_t^T\langle 2|\nabla(u(s)-\hat{\xi}(s))|,|\hat{g}(s)|+|g_0(s)|\rangle \,ds|\mathcal{F}_t\right]\\
\leq & 2\epsilon E\left[\int_t^T\|\nabla (u(s)-\hat{\xi}(s))\|^2\,ds|\mathcal{F}_t\right]+C(\epsilon)E\left[\int_t^T\|g_0(s)\|^2\,ds|\mathcal{F}_t\right]+C(\epsilon)E\left[\int_t^T\|\hat{g}(s)\|^2\,ds|\mathcal{F}_t\right]\\
&+E\left[\int_t^T\langle 2|\nabla(u(s)-\hat{\xi}(s))|,L|u(s)-\hat{\xi}(s)|+L|\hat{\xi}(s)|\rangle\, ds|\mathcal{F}_t\right]\\
&+E\left[\int_t^T\langle 2|\nabla(u(s)-\hat{\xi}(s))|,\frac{\kappa}{2}|\nabla(u(s)-\hat{\xi}(s))|+\frac{\kappa}{2}|\nabla\hat{\xi}(s)|\rangle\, ds|\mathcal{F}_t\right]\\
&+E\left[\int_t^T\langle 2|\nabla(u(s)-\hat{\xi}(s))|,\sqrt{\beta}|v(s)-\hat{v}(s)|+\sqrt{\beta}|\hat{v}(s)|\rangle\, ds|\mathcal{F}_t\right]\\
\leq & 2\epsilon E\left[\int_t^T\|\nabla(u(s)-\hat{\xi}(s))\|^2\,ds|\mathcal{F}_t\right]+C(\epsilon)E\left[\int_t^T\|\hat{g}(s)\|\,ds|\mathcal{F}_t\right]+C(\epsilon)E\left[\int_t^T\|g_0(s)\|\,ds|\mathcal{F}_t\right]\\
&+2\epsilon E\left[\int_t^T\|\nabla(u(s)-\hat{\xi}(s))\|^2\, ds|\mathcal{F}_t\right]+C(\epsilon,L)E\left[\int_t^T\|u(s)-\hat{\xi}(s)\|^2\, ds|\mathcal{F}_t\right]+C(\epsilon,L)E\left[\int_t^T\|\hat{\xi}(s)\|^2\, ds|\mathcal{F}_t\right]\\
&+\kappa E\left[\int_t^T\|\nabla(u(s)-\hat{\xi}(s))\|^2\,ds|\mathcal{F}_t\right]+\epsilon E\left[\int_t^T\|\nabla(u(s)-\hat{\xi}(s))\|^2\,ds|\mathcal{F}_t\right]+C(\epsilon,\kappa)E\left[\int_t^T\|\nabla\hat{\xi}(s))\|^2\,ds|\mathcal{F}_t\right]\\
&+\beta\theta E\left[\int_t^T\|\nabla(u(s)-\hat{\xi}(s))\|^2\,ds|\mathcal{F}_t\right]+\frac{1}{\theta}E\left[\int_t^T\|v(s)-\hat{v}(s)\|^2\,ds|\mathcal{F}_t\right]\\
&+\epsilon E\left[\int_t^T\|\nabla(u(s)-\hat{\xi}(s))\|^2\,ds|\mathcal{F}_t\right]+C(\epsilon,\beta)E\left[\int_t^T\|\hat{v}(s)\|^2\,ds|\mathcal{F}_t\right]\\
\leq & \left(6\epsilon +\kappa+\beta\theta\right)E\left[\int_t^T\|\nabla(u-\hat{\xi})(s)\|^2\,ds|\mathcal{F}_t\right]+\frac{1}{\theta}E\left[\int_t^T\|v(s)-\hat{v}(s)\|^2\,ds|\mathcal{F}_t\right]\\   &+C(\epsilon)E\left[\int_t^T\|g_0(s)\|^2\,ds|\mathcal{F}_t\right]+C(\epsilon)E\left[\int_t^T\|\hat{g}(s)\|^2\,ds|\mathcal{F}_t\right]\\
   & + C(\epsilon,L)E\left[\int_t^T\|u(s)-\hat{\xi}(s)\|^2\,ds|\mathcal{F}_t\right]+C(\epsilon ,L)E\left[\int_t^T\|\hat{\xi}(s)\|^2\,ds|\mathcal{F}_t\right]\\
   &+C(\epsilon, \kappa)E\left[\int_t^T\|\nabla\hat{\xi}(s)\|^2\,ds|\mathcal{F}_t\right]+C(\epsilon,\beta)E\left[\int_t^T\|\hat{v}(s)\|^2\,ds|\mathcal{F}_t\right].
\end{align*}

It follows from $(\mathcal{A}_3)$ that:
\begin{align*}
I_3:&=E\left[\int_t^T\langle 2(u(s)-\hat{\xi}(s)),f(s,u(s),\nabla u(s),v(s))-\hat{f}(s)\rangle \,ds|\mathcal{F}_t\right]\\
&\leq  E\left[\int_t^T\langle 2|u(s)-\hat{\xi}(s)|,|f_0|+L|u(s)|+L|\nabla u(s)|+L|v(s)|+|\hat{f}(s)|\rangle \,ds|\mathcal{F}_t\right].
\end{align*}
In view of (\ref{repeatedestimate}) it further holds for each $\epsilon_1>0$ that:
\begin{align*}
& E\left[\int_t^T\langle 2|u(s)-\hat{\xi}(s)|,L|u(s)|+L|\nabla u(s)|+L|v(s)|\rangle \,ds|\mathcal{F}_t\right]\\
\leq & \epsilon_1 E\left[\int_t^T\|\nabla u(s)\|^2\,ds|\mathcal{F}_t\right]+\epsilon_1 E\left[\int_t^T\|v(s)\|^2\,ds|\mathcal{F}_t\right]+C(\epsilon_1,L)E\left[\int_t^T\|u(s)-\hat{\xi}(s)\|^2\,ds|\mathcal{F}_t\right]\\
&+E\left[\int_t^T\langle 2|(u-\hat{\xi})(s)|,L|(u-\hat{\xi})(s)|+L|\hat{\xi}(s)|\rangle \,ds|\mathcal{F}_t\right]\\
\leq&  2\epsilon_1 E\left[\int_t^T\|\nabla (u(s)-\hat{\xi}(s))\|^2\,ds|\mathcal{F}_t\right]+2\epsilon _1 E\left[\int_t^T\|v(s)-\hat{v}(s)\|^2ds|\mathcal{F}_t\right]+2\epsilon _1E\left[\int_t^T\|\hat{v}(s)\|^2\,ds|\mathcal{F}_t\right]\\
&+2\epsilon _1E\left[\int_t^T\|\nabla\hat{\xi}(s)\|^2\,ds|\mathcal{F}_t\right]
+C(\epsilon_1,L)E\left[\int_t^T\|u(s)-\hat{\xi}(s)\|^2\,ds|\mathcal{F}_t\right]+E\left[\int_t^T\|\hat{\xi}(s)\|^2\,ds|\mathcal{F}_t\right],
\end{align*}
and by the H\"older inequality one has that:
\begin{align*}
& E\left[\int_t^T\langle 2|u(s)-\hat{\xi}(s)|,|f_0|+|\hat{f}(s)|\rangle \,ds|\mathcal{F}_t\right]\\
&\leq 2E\left[\int_t^T\|(u-\hat{\xi})(s)\|^2\,ds|\mathcal{F}_t\right]+E\left[\int_t^T\|f_0(s)\|^2\,ds|\mathcal{F}_t\right]+E\left[\int_t^T\|\hat{f}(s)\|^2\,ds|\mathcal{F}_t\right]
\end{align*}
In addition,
$$I_4:=E\left[\|G-\hat{G}\|^2|\mathcal{F}_t\right]\leq  \textrm{esssup}_{\omega \in \Omega}\|G-\hat{G}\|^2.$$

Summing up the estimates $I_1$-$I_4$ and taking the supremum w.r.t. $(\omega,s)\in \Omega\times[t,T]$ on both sides we arrive at:
\begin{align*}
&\|u-\hat{\xi}\|^2_{\infty,2;\mathcal{O}_t}+\|v-\hat{v}\|^2_{0,2;\mathcal{O}_t}+\left(\lambda -\kappa-\beta\theta-6\epsilon-2\epsilon_1\right)\|\nabla(u-\hat{\xi})\|^2_{0,2;\mathcal{O}_t}\\
\leq \,& \left(\frac{1}{\varrho}+\frac{1}{\theta}+2\epsilon_1\right)\|v-\hat{v}\|^2_{0,2;\mathcal{O}_t}
+C(\epsilon,\epsilon_1,L)\int_t^T\|u-\hat{\xi}\|^2_{\infty,2;\mathcal{O}_s}\,ds+C(\epsilon,\epsilon_1,\beta)\|\hat{v}\|^2_{0,2;\mathcal{O}_t}\\
&+C(\epsilon,\epsilon_1,\kappa,L)\|\hat{\xi}\|^2_{\mathcal{V}_2(\mathcal{O}_t)}+\|f_0\|^2_{0,2;\mathcal{O}_t}+\|\hat{f}\|^2_{0,2;\mathcal{O}_t}
+C(\epsilon)\left(\|g_0\|^2_{0,2;\mathcal{O}_t}+\|\hat{g}\|^2_{0,2;\mathcal{O}_t}\right)\\
&+\textrm{esssup}_{\omega\in\Omega}\|G-\hat{G}\|^2.
\end{align*}
By assumption ($\mathcal{A}_2$) we can choose $\theta>\varrho'$ such that $\lambda-\kappa-\beta\theta>0$, and $\theta>\varrho'$ also implies $\frac{1}{\varrho}+\frac{1}{\theta}<1$. Now taking $\epsilon$ and $\epsilon_1$ small enough such that $\lambda -\kappa-\beta\theta-6\epsilon-2\epsilon_1>0$ and $\frac{1}{\varrho}+\frac{1}{\theta}+2\epsilon_1<1$, we have
\begin{equation}\label{equation3.7}
\begin{split}
&\|u-\hat{\xi}\|^2_{\mathcal{V}^2(\mathcal{O}_t)}+\|v-\hat{v}\|^2_{0,2;\mathcal{O}_t}\\
\leq &\, C(\epsilon,\epsilon_1,\lambda,\beta,\kappa,L,\varrho)\left(\int_t^T\|u-\hat{\xi}\|^2_{\infty,2;\mathcal{O}_s}\,ds+\|\hat{v}\|^2_{0,2;\mathcal{O}_t}+\|\hat{\xi}\|^2_{\mathcal{V}_2(\mathcal{O}_t)}\right.\\
&+B_2(f_0,g_0;\mathcal{O}_t)^2+B_2(\hat{f},\hat{g};\mathcal{O}_t)^2+\textrm{esssup}_{\omega\in\Omega}\|G-\hat{G}\|^2_{L^2(\mathcal{O})}\bigg).
\end{split}
\end{equation}
By Gronwall's inequality,
\begin{equation}\label{equation3.8}
\begin{split}
&\|u-\hat{\xi}\|^2_{\mathcal{V}_2(\mathcal{O}_t)}+\|v-\hat{v}\|^2_{0,2;\mathcal{O}_t}\\
\leq &C(\epsilon,\epsilon_1,\lambda,\beta,\kappa,L,\varrho,T)\left(\|\hat{v}\|^2_{0,2;\mathcal{O}_t}
+\|\hat{\xi}\|^2_{\mathcal{V}_2(\mathcal{O}_t)}+\textrm{esssup}_{\omega\in\Omega}\|G-\hat{G}\|^2_{L^2(\mathcal{O}_t)}\right.\\
&\left.+B_2(f_0,g_0;\mathcal{O}_t)^2+B_2(\hat{f},\hat{g};\mathcal{O}_t)^2\right).
\end{split}
\end{equation}
Since $\hat{\xi}|_{\partial\mathcal{O}}=0$, we can apply Proposition \ref{proposition6.1} to $\|\hat{\xi}(t)\|^2$. Starting from (\ref{estimate_begin}), using similar estimates,
 \begin{equation}\label{estimate_xi}
 \|\hat{\xi}\|^2_{\mathcal{V}_2(\mathcal{O}_t)}+\|\hat{v}\|^2_{0,2;\mathcal{O}_t}
\leq  C\left(B_2(\hat{f},\hat{g};\mathcal{O}_t)^2+\textrm{esssup}_{\omega\in\Omega}\|\hat{G}\|^2\right),
 \end{equation}
 where $C$ only depends on $\lambda$, $\beta$, $\kappa$, $\varrho$, $L$ and $T$. The estimate
 (\ref{equation3.8}) together with (\ref{estimate_xi}) yields (\ref{equation3.2}).
\end{proof}


With the same notation as in Theorem \ref{theorem3.2}, we can relax the zero Dirichlet boundary condition in Theorem \ref{theorem3.2} by assuming $u|_{\partial\mathcal{O}}=\tilde{u}|_{\partial\mathcal{O}}$ for some
$(\tilde{u},\tilde{v})\in\mathcal{U}(-\infty,\tilde{f},\tilde{g},\tilde{G})$
where the coefficients $a$, $\sigma$, $\tilde{f}$, $\tilde{g}$ and $\tilde{G}$ satisfy $(\mathcal{A}_2)$ and $(\mathcal{A}_3)$ respectively, and $\tilde{f}$ and $\tilde{g}$ do not depend on $\tilde{u}$, $\nabla \tilde{u}$ and $\tilde{v}$.
Assume further that $\hat{\xi}|_{\partial\mathcal{O}}\leq \tilde{u}|_{\partial{\mathcal{O}}}$ and put $\bar{\xi}:=\hat{\xi}-\tilde{u}$. Then, $(\bar{\xi},\bar{v})\in\mathcal{U}(-\infty,\bar{f},\bar{g},\bar{G})$,
where $\bar{v}=\hat{v}-\tilde{v}$, $\bar{f}=\hat{f}-\tilde{f}$, $\bar{g}=\hat{g}-\tilde{g}$ and $\bar{G}=\hat{G}-\tilde{G}$. Suppose now that $(\breve{\xi},\breve{v})\in\mathcal{U}(-\infty,\bar{f},\bar{g},\bar{G})$ with $\breve{\xi}|_{\partial_{\mathcal{O}}}=0$.
Then, $\breve{\xi}|_{\partial\mathcal{O}}=0\geq (\hat{\xi}-\tilde{u})|_{\partial\mathcal{O}}=\bar{\xi}|_{\partial\mathcal{O}}$ and the maximum principle in Lemma \ref{MP-general-domain-quasilinear-BSPDE} yields $\breve{\xi}\geq \hat{\xi}-\tilde{u}\geq \xi-\tilde{u}$. Therefore, our RBSPDE \eqref{equation1.1} is equivalent to the following one but with zero-Dirichlet condition:
\begin{equation}\label{equivalentequation}
\left\{\begin{array}{ll}
\begin{split}
-d\breve{u}(t,x)&=\left[\partial _j(a^{ij}\partial _i\breve{u}+\sigma ^{jr}\breve{v}^r)(t,x)
+(f+\nabla \cdot g)(t,x,\breve{u}+\tilde u,\nabla (\breve{u} + \tilde u),\breve{v}+\tilde v) \right. \\
&	\left. -(\tilde{f}+\nabla \cdot \tilde{g})(t,x)\right]\,dt +\mu (dt,x)-\breve{v}^r(t,x)dW^r_t,~~~(t,x)\in Q;\\
\breve{u}(T,x)&=G(x)-\tilde{G}(x),~~~x\in \mathcal{O};\\
\breve{u} &\geq \xi-\tilde{u},\quad \mathbb{P}\otimes dt\otimes dx\text{-a.e.};\\
\int_Q(\breve{u}&-(\xi -\tilde{u}))(t,x)\,\mu (dt,dx)=0.
\end{split}
\end{array}
\right.
\end{equation}
By Theorem \ref{theorem3.2}, there is a unique solution $u-\tilde{u}$ to the RBSPDE (\ref{equivalentequation}) satisfying zero-Dirichlet condition. In this way, Theorem \ref{theorem3.2} extends to RBSPDEs with general Dirichlet conditions.


\section{Maximum Principle for RBSPDE}

In this section we state and prove our maximum principles for RBSPDEs. We start with a global maximum principle on general domains, which states that the weak solution $u$ is bounded on the whole domain if it is bounded on the parabolic boundary. Subsequently we analyze the local behavior of $u^{\pm}$ when $u$ is not necessarily bounded on the parabolic boundary.

\subsection{Global Case}

This section establishes a maximum principle for the RBSPDE (\ref{equation1.1}) on a general domain $\mathcal{O}$. Since the Lebesgue measure of $\mathcal{O}$ might not be bounded, the scheme in \cite{qiutang} cannot be applied. Instead, motivated by \cite{qiu-2015-DSPDE}, we use a stochastic De Girogi's scheme that is independent of the measure of the domain. In what follows $\partial_pQ=(\{T\}\times\mathcal{O})\cup([0,T]\times\partial\mathcal{O})$ denotes the parabolic boundary of $Q$.


\begin{theorem}\label{Theorem_MP_GD}
\begin{itemize}
\item[(1)] Assume that $(\mathcal{A}_1)$-$(\mathcal{A}_5)$ hold. If the triplet $(u,v,\mu)$ is a solution to the RBSPDE (\ref{equation1.1}), then
\begin{equation*} 
\begin{split}
&\textrm{esssup}_{(\omega,t,x)\in\Omega \times Q}u^{\pm} \\
\leq &  C\left(\textrm{esssup}_{(\omega,t,x)\in\Omega \times \partial _pQ}u^{\pm}+\textrm{esssup}_{(\omega,t,x)\in\Omega \times \partial _pQ}\hat{\xi}^{\pm}\right. \\
&\quad \left.+A_p(f_0^{\pm},g_0;Q)+B_2(f_0^{\pm},g_0;Q)+A_p(\hat{f}^{\pm},\hat{g};Q)+B_2(\hat{f}^{\pm},\hat{g};Q)
 \right),
\end{split}
\end{equation*}
where the constant $C$ depends only on $\lambda$, $\kappa$, $\beta$, $L$, $\varrho$, $T$, $p$ and $n$.
\item[(2)] If all conditions in (1) hold except assumption $(\mathcal{A}_5)$ is changed to
\begin{align}\label{additional-assumption-f-g}
g(t,x,r,0,0)=g(t,x,0,0)~\textrm{and}~f(t,x,r,0,0)~ \textrm{is}~ \textrm{non-increasing} ~w.r.t.~ \textrm{r},
\end{align}
then
\begin{equation*}
\begin{split}
&\textrm{esssup}_{(\omega,t,x)\in\Omega\times Q} u^{\pm} \\
\leq & \textrm{esssup}_{(\omega,t,x)\in\Omega\times\partial_p Q}u^{\pm}+\textrm{esssup}_{(\omega,t,x)\in\Omega\times\partial_p Q}\hat{\xi}^{\pm} \\
&+C\left(A_p(f_0^{\pm},g_0;Q)^{\frac{np}{np+2(p-n-2)}}B_2(f_0^{\pm},g_0;Q)^{\frac{2(p-n-2)}{np+2(p-n-2)}}\right. \\
&\left.+A_p(\hat{f}^{\pm},\hat{g};Q)^{\frac{np}{np+2(p-n-2)}}B_2(\hat{f}^{\pm},\hat{g};Q)^{\frac{2(p-n-2)}{np+2(p-n-2)}}\right),
\end{split}
\end{equation*}
where $C$ depends on $\lambda$, $\kappa$, $\beta$, $L$, $\varrho$, $n$, $p$ and $T$.
\end{itemize}
\end{theorem}
\begin{proof}
We only consider the estimate for the positive part $u^+$. The one for the negative part $u^-$ follows analogously. Further, we may w.l.o.g.~assume that $f(t,x,r,0,0)$ is non-increasing in $r$. Otherwise, the desired maximum principle can be derived from the maximum principle for the RBSPDE
\begin{equation*}
\left\{
\begin{array}{l}
\begin{aligned}
\label{u-bar}
-d\bar{u}(t,x)&=\left[\partial_j(a^{ij}(t,x)\partial_i\bar{u}(t,x)+\sigma^{jr}(t,x)\bar{v}^r(t,x))+\bar{f}(t,x,\bar{u}(t,x),\nabla\bar{u}(t,x),\bar{v}(t,x))\right.\\
              &\left.+\nabla\cdot\bar{g}(t,x,\bar{u}(t,x),\nabla\bar{u}(t,x),\bar{v}(t,x))\right]\,dt+\bar{\mu}(dt,x)-\bar{v}^r(t,x)\,dW^r_t,\\
\bar{u}(T,x)&=\bar{G}(x),\\
\bar{u}(t,x)&\geq\bar{\xi}(t,x)~dt\times dx\times d\mathbb{P}-a.e.,\\
\int_Q(\bar{u}(t,x)&-\bar{\xi}(t,x))\,\bar{\mu}(dt,dx)=0,
\end{aligned}
\end{array}
\right.
\end{equation*}
where $\bar{u}(t,x)=e^{Lt}u(t,x)$, $\bar{v}(t,x)=e^{Lt}v(t,x)$, $\bar{\mu}(dt,dx)=e^{Lt}\mu(dt,dx)$, $\bar{G}(x)=e^{LT}G(x)$, $\bar{\xi}(t,x) = e^{Lt}\xi(t,x)$ and
\begin{align*}
	\bar{f}(t,x,\bar{u}(t,x),\nabla\bar{u}(t,x),\bar{v}(t,x)) & =e^{Lt}f(t,x,e^{-Lt}\bar{u}(t,x),e^{-Lt}\nabla\bar{u}(t,x),e^{-Lt}\bar{v}(t,x))-L\bar{u}(t,x)\\
	\bar{g}(t,x,\bar{u}(t,x),\nabla\bar{u}(t,x),\bar{v}(t,x)) & =e^{Lt}g(t,x,e^{-Lt}\bar{u}(t,x),e^{-Lt}\nabla\bar{u}(t,x),e^{-Lt}\bar{v}(t,x)).
\end{align*}


Now, for $t\in[0,T]$ define
\[
	\bar{k}=\textrm{esssup}_{(\omega,t,x)\in\Omega\times\partial_{p}\mathcal{O}_t}u^++\textrm{esssup}_{(\omega,t,x)\in\Omega\times\mathcal{O}_t}\hat{\xi}^+.
\]
For a positive constant $k$ to be determined later and each $m\in \mathbb{N}_0$, let $\bar{k}_m=k(1-2^{-m})$ and $k_m=\bar{k}_m+\bar{k}$.
By Theorem \ref{Ito-u+}, for $m\geq 1$,
\begin{equation}\label{Ito-for-u-k_m}
\begin{split}
&\|(u-k_m)^+(t)\|^2+\int_t^T\|v^{k_m}(s)\|^2ds
\\
=&-2\int_t^T\langle\partial_j(u-k_m)^+(s),a^{ij}\partial_i(u-k_m)^+(s)+\sigma ^{jr}(s)v^{k_m,r}(s)\rangle\, ds\\
&-2\int_t^T\langle\partial_j(u-k_m)^+(s),g^{j,k_m}\left(s,(u-k_m)^+(s),\nabla u(s),v^{k_m}(s)\right)\rangle\, ds\\
&+2\int_t^T\langle (u-k_m)^+(s),f^{k_m}\left(s,(u-k_m)^+(s),\nabla u(s),v^{k_m}(s)\right)\rangle \,ds\\
&+2\int_{\mathcal{O}_t}(u-k_m)^+(s,x)\,\mu(ds,dx)-2\int_t^T\langle (u-k_m)^+(s),v^{r,k_m}(s)\,dW^r_s\rangle,
\end{split}
\end{equation}
where $v^{r,k_m}:=v^r1_{\{u>k_m\}}$, $f^{k_m}(\cdot,\cdot,\cdot,X,\cdot,\cdot):=f(\cdot,\cdot,\cdot,X+k_m,\cdot,\cdot)$, $g^{j,k_m}(\cdot,\cdot,\cdot,X,\cdot,\cdot):=g^j(\cdot,\cdot,\cdot,X+k_m,\cdot,\cdot)$. All terms in (\ref{Ito-for-u-k_m}) are well defined. In particular, the stochastic integral is in fact a martingale.
 Taking conditional expectations on both sides w.r.t. $\mathcal{F}_t$ yields the following estimates for the remaining terms. Similar estimates as for $I_1$ in the proof of Theorem \ref{theorem3.2} yield,
\begin{equation}\label{J_1}
\begin{split}
J_1:=&-2E\left[\int_t^T\langle\partial_j(u-k_m)^+(s),a^{ij}\partial_i(u-k_m)^+(s)+\sigma ^{jr}(s)v^{r,k_m}(s)\rangle \, ds|\mathcal{F}_t\right]\\
\leq &-\lambda E\left[\int_t^T\|\nabla(u-k_m)^+(s)\|^2\,ds|\mathcal{F}_t\right]+\frac{1}{\varrho}E\left[\int_t^T\|v^{k_m}(s)\|^2\,ds|\mathcal{F}_t\right].
\end{split}
\end{equation}
By analogy to the estimate of $I_2$, for each $\epsilon>0$ and $\theta>0$, we have that
\begin{equation}\label{J_2-1}
\begin{split}
J_2:=&-2E\left[\int_t^T\langle\partial_j(u-k_m)^+(s),g^{j,k_m}\left(s,(u-k_m)^+(s),\nabla u(s),v^{k_m}(s)\right)\rangle \, ds|\mathcal{F}_t\right]  \\
\leq &2E\left[\int_t^T\langle|\nabla(u-k_m)^+(s)|,|g^{k_m}_0|+L|(u-k_m)^+(s)|+\frac{\kappa}{2}|\nabla (u-k_m)^+(s)| \right. \\
& \left. +\sqrt{\beta}|v^{k_m}(s)|\rangle \,ds|\mathcal{F}_t\right]\\
\leq &(\kappa+\beta\theta+\epsilon)E\left[\int_t^T\|\nabla(u-k_m)^+(s)\|^2\,ds|\mathcal{F}_t\right]+\frac{1}{\theta}E\left[\int_t^T\|v^{k_m}(s)\|^2\,ds|\mathcal{F}_t\right]\\
 &+\frac{L^2}{\epsilon}E\left[\int_t^T\|(u-k_m)^+(s)\|^2\,ds|\mathcal{F}_t\right]+2E\left[\int_t^T\langle|\nabla(u-k_m)^+(s)|,|g_0^{k_m}(s)|\rangle \,ds|\mathcal{F}_t\right].
\end{split}
\end{equation}

From
\begin{align*}
\left[(u-k_{m-1})^+-(u-k_m)^+\right]1_{\{u>k_m\}}=(k_m-k_{m-1})1_{\{u>k_m\}}=2^{-m}k1_{\{u>k_m\}}
\end{align*}
we get
\begin{align}\label{estimate_indicator}
1_{\{u>k_m\}}\leq\frac{2^m(u-k_{m-1})^+}{k}1_{\{u>k_m\}}\leq\frac{2^m(u-k_{m-1})^+}{k}.
\end{align}
By (\ref{estimate_indicator}) and  ($\mathcal{A}_5$) it holds that:
\begin{equation}\label{J_2-2}
\begin{split}
&E\left[\int_t^T\langle|\nabla(u-k_m)^+(s)|,|g^{k_m}_0(s)|\rangle\, ds|\mathcal{F}_t\right] \\
\leq & \left(E\left[\int_t^T\|\nabla(u-k_m)^+(s)\|^2\,ds|\mathcal{F}_t\right]\right)^{\frac{1}{2}}\left(E\left[\int_t^T\int_{\mathcal{O}}|g^{k_m}_0(s,x)1_{\{u>k_m\}}|^2\,dxds|\mathcal{F}_t\right]\right)^{\frac{1}{2}} \\
\leq & \|\nabla (u-k_m)^+\|_{0,2;\mathcal{O}_t}\|g^{k_m}_0\|_{0,p;\mathcal{O}_t}\left(E\left[\int_t^T\int_{\mathcal{O}}1_{\{u>k_m\}}\,dxds|\mathcal{F}_t\right]\right)^{\frac{1}{2}-\frac{1}{p}} \\
\leq & \|\nabla (u-k_m)^+\|_{0,2;\mathcal{O}_t}\|g^{k_m}_0\|_{0,p;\mathcal{O}_t}\left(E\left[\int_t^T\int_{\mathcal{O}}\left(\frac{2^m(u-k_{m-1})^+}{k}\right)^{\frac{2(n+2)}{n}}\,dxds|\mathcal{F}_t\right]\right)^{\frac{1}{2}-\frac{1}{p}} \\
\leq & \left(\frac{2^m}{k}\right)^{1+\frac{2(p-n-2)}{np}}\|\nabla (u-k_m)^+\|_{0,2;\mathcal{O}_t}\|g^{k_m}_0\|_{0,p;\mathcal{O}_t}\|(u-k_{m-1})^+\|_{0,\frac{2(n+2)}{n};\mathcal{O}_t}^{1+\frac{2(p-n-2)}{np}} \\
\leq & \left(\frac{2^m}{k}\right)^{1+\frac{2(p-n-2)}{np}}\|\nabla (u-k_{m-1})^+\|_{0,2;\mathcal{O}_t}\|g^{k_m}_0\|_{0,p;\mathcal{O}_t}\|(u-k_{m-1})^+\|_{0,\frac{2(n+2)}{n};\mathcal{O}_t}^{1+\frac{2(p-n-2)}{np}} \\
\leq & \left(\frac{2^m}{k}\right)^{1+\frac{2(p-n-2)}{np}}\|\nabla (u-k_{m-1})^+\|_{0,2;\mathcal{O}_t}(\|g_0\|_{0,p;\mathcal{O}_t}+Lk_m)\|(u-k_{m-1})^+\|_{0,\frac{2(n+2)}{n};\mathcal{O}_t}^{1+\frac{2(p-n-2)}{np}}.\\
\end{split}
\end{equation}
Combining (\ref{J_2-1}) and (\ref{J_2-2}), we see that
\begin{equation}
\begin{split}
J_2\leq & (\kappa+\beta\theta+\epsilon)E\left[\int_t^T\|\nabla(u-k_m)^+(s)\|^2\,ds|\mathcal{F}_t\right]+\frac{1}{\theta}E\left[\int_t^T\|v^{k_m}(s)\|^2\,ds|\mathcal{F}_t\right]\\
 &+\frac{L^2}{\epsilon}E\left[\int_t^T\|(u-k_m)^+(s)\|^2\,ds|\mathcal{F}_t\right]\\
 &+2\left(\frac{2^m}{k}\right)^{1+\frac{2(p-n-2)}{np}}\|\nabla (u-k_{m-1})^+\|_{0,2;\mathcal{O}_t}(\|g_0\|_{0,p;\mathcal{O}_t}+Lk_m)\|(u-k_{m-1})^+\|_{0,\frac{2(n+2)}{n};\mathcal{O}_t}^{1+\frac{2(p-n-2)}{np}}.\\
\end{split}
\end{equation}

For each $\epsilon_1>0$, by (\ref{repeatedestimate}) the monotonicity of $f(t,x,r,0,0)$ yields:
\begin{equation}\label{J_3-1}
\begin{split}
J_3:=& 2E\left[\int_t^T\langle (u-k_m)^+(s),f^{k_m}\left(s,(u-k_m)^+(s),\nabla u(s),v^{k_m}(s)\right)\rangle \, ds|\mathcal{F}_t\right]
\\
\leq & 2E\left[\int_t^T\langle(u-k_m)^+(s),f_0^{k_m}(s)+L(u-k_m)^+(s)+L\nabla(u-k_m)^+(s)+L|v^{k_m}(s)|\rangle \, ds|\mathcal{F}_t\right] \\
\leq & 2E\left[\int_t^T\langle(u-k_m)^+(s),f_0(s)+L(u-k_m)^+(s)+L\nabla(u-k_m)^+(s)+L|v^{k_m}(s)|\rangle \, ds|\mathcal{F}_t\right]\\
\leq & \left(2L+\frac{2L^2}{\epsilon_1}\right)E\left[\int_t^T\|(u-k_m)^+(s)\|^2\,ds|\mathcal{F}_t\right]+\epsilon_1E\left[\int_t^T\|\nabla(u-k_m)^+(s)\|^2\,ds|\mathcal{F}_t\right] \\
&+\epsilon_1E\left[\int_t^T\|v^{k_m}(s)\|^2\,ds|\mathcal{F}_t\right]
+2E\left[\int_t^T\langle(u-k_m)^+(s),f_0(s)\rangle \,ds|\mathcal{F}_t\right].
\end{split}
\end{equation}

By (\ref{estimate_indicator}) again, we have
\begin{equation}\label{J_3-2}
\begin{split}
&E\left[\int_t^T\langle (u-k_m)^+(s),f_0(s)\rangle \,ds|\mathcal{F}_t\right] \\
\leq & E\left[\int_t^T\langle (u-k_m)^+(s),f_0^+(s)\rangle \,ds|\mathcal{F}_t\right] \\
\leq & \left(E\left[\int_t^T\int_{\mathcal{O}}|(u-k_m)^+|^{\frac{2(n+2)}{n}}\,dxds|\mathcal{F}_t\right]\right)^{\frac{n}{2(n+2)}}\left(E\left[\int_t^T\int_{\mathcal{O}}|f_0^+(s,x)|^{\frac{p(n+2)}{p+n+2}}\,dxds|\mathcal{F}_t\right]\right)^{\frac{p+n+2}{p(n+2)}} \\
&\times\left(E\left[\int_t^T\int_{\mathcal{O}}1_{\{u>k_m\}}\,dxds|\mathcal{F}_t\right]\right)^{\frac{1}{2}-\frac{1}{p}} \\
\leq & \|(u-k_m)^+\|_{0,\frac{2(n+2)}{n};\mathcal{O}_t}\|f_0^+\|_{0,\frac{p(n+2)}{p+n+2};\mathcal{O}_t}\left(E\left[\int_t^T\int_{\mathcal{O}}\left(\frac{2^m(u-k_{m-1})^+}{k}\right)^{\frac{2(n+2)}{n}}\,dxds|\mathcal{F}_t\right]\right)^{\frac{1}{2}-\frac{1}{p}} \\
\leq & \left(\frac{2^m}{k}\right)^{1+\frac{2(p-n-2)}{np}}\|(u-k_m)^+\|_{0,\frac{2(n+2)}{n};\mathcal{O}_t}\|(u-k_{m-1})^+\|_{0,\frac{2(n+2)}{n};\mathcal{O}_t}^{1+\frac{2(p-n-2)}{np}}\|f_0^+\|_{0,\frac{p(n+2)}{p+n+2};\mathcal{O}_t} \\
\leq & \left(\frac{2^m}{k}\right)^{1+\frac{2(p-n-2)}{np}}\|(u-k_{m-1})^+\|_{0,\frac{2(n+2)}{n};\mathcal{O}_t}^{2+\frac{2(p-n-2)}{np}}\|f^+_0\|_{0,\frac{p(n+2)}{p+n+2};\mathcal{O}_t}.
\end{split}
\end{equation}
Therefore, by (\ref{J_3-1}) and (\ref{J_3-2}) we conclude
\begin{equation}
\begin{split}
J_3\leq & \left(2L+\frac{2L^2}{\epsilon_1}\right)E\left[\int_t^T\|(u-k_m)^+(s)\|^2\,ds|\mathcal{F}_t\right]+\epsilon_1E\left[\int_t^T\|\nabla(u-k_m)^+(s)\|^2\,ds|\mathcal{F}_t\right] \\
&+\epsilon_1E\left[\int_t^T\|v^{k_m}(s)\|^2\,ds|\mathcal{F}_t\right]
+2\left(\frac{2^m}{k}\right)^{1+\frac{2(p-n-2)}{np}}\|(u-k_{m-1})^+\|_{0,\frac{2(n+2)}{n};\mathcal{O}_t}^{2+\frac{2(p-n-2)}{np}}\|f^+_0\|_{0,\frac{p(n+2)}{p+n+2};\mathcal{O}_t}.
\end{split}
\end{equation}

Finally, note that
\begin{align*}
\int_t^T\int_{\mathcal{O}}(u-k_m)^+\,\mu(dxds) \leq \int_t^T\int_{\mathcal{O}}(u-\xi)^+\,\mu(dxds)+\int_t^T\int_{\mathcal{O}}(\xi-\hat{\xi}^+)^+\,\mu(dxds)=0.
\end{align*}
Combining the above estimates, we get
\begin{align*}
&\|(u-k_m)^+(t)\|^2+E\left[\int_t^T\|v^{k_m}(s)\|^2\,ds|\mathcal{F}_t\right]\nonumber\\
\leq & (-\lambda+\kappa+\beta\theta+\epsilon+\epsilon_1) E\left[\int_t^T\|\nabla(u-k_m)^+\|^2\,ds|\mathcal{F}_t\right]\nonumber\\
&+\left(\frac{1}{\theta}+\frac{1}{\varrho}+\epsilon_1\right)E\left[\int_t^T\|v^{k_m}\|^2\,ds|\mathcal{F}_t\right]
+\left(2L+\frac{L^2}{\epsilon}+\frac{2L^2}{\epsilon_1}\right)E\left[\int_t^T\|(u-k_m)^+\|^2\,ds|\mathcal{F}_t\right]\nonumber\\
&+2\left(\frac{2^m}{k}\right)^{1+\frac{2(p-n-2)}{np}}\|\nabla (u-k_{m-1})^+\|_{0,2;\mathcal{O}_t}\|(u-k_{m-1})^+\|_{0,\frac{2(n+2)}{n};\mathcal{O}_t}^{1+\frac{2(p-n-2)}{np}}\left(\|g_0\|_{0,p;\mathcal{O}_t}+Lk_m\right)\nonumber\\
&+2\left(\frac{2^m}{k}\right)^{1+\frac{2(p-n-2)}{np}}\|(u-k_{m-1})^+\|_{0,\frac{2(n+2)}{n};\mathcal{O}_t}^{2+\frac{2(p-n-2)}{np}}\|f_0^{+}\|_{0,\frac{p(n+2)}{p+n+2};\mathcal{O}_t}.\nonumber\\
\end{align*}
From this it is straightforward to see that
\begin{align*}
&\min\{1,\lambda-\kappa-\beta\theta-\epsilon-\epsilon_1\}\left\{\|(u-k_m)^+(t)\|^2+E\left[\int_t^T\|\nabla(u-k_m)^+\|^2\,ds|\mathcal{F}_t\right]\right\}\nonumber\\
&+\left(1-\frac{1}{\theta}-\frac{1}{\varrho}-\epsilon_1\right)E\left[\int_t^T\|v^{k_m}\|^2\,ds|\mathcal{F}_t\right]\nonumber\\
\leq
&
\left(2L+\frac{L^2}{\epsilon}+\frac{2L^2}{\epsilon_1}\right)E\left[\int_t^T\|(u-k_m)^+\|^2\,ds|\mathcal{F}_t\right]\nonumber\\
&+2\left(\frac{2^m}{k}\right)^{1+\frac{2(p-n-2)}{np}}\|\nabla (u-k_{m-1})^+\|_{0,2;\mathcal{O}_t}\|(u-k_{m-1})^+\|_{0,\frac{2(n+2)}{n};\mathcal{O}_t}^{1+\frac{2(p-n-2)}{np}}\left(\|g_0\|_{0,p;\mathcal{O}_t}+Lk_m\right)\nonumber\\
&+2\left(\frac{2^m}{k}\right)^{1+\frac{2(p-n-2)}{np}}\|(u-k_{m-1})^+\|_{0,\frac{2(n+2)}{n};\mathcal{O}_t}^{2+\frac{2(p-n-2)}{np}}\|f_0^{+}\|_{0,\frac{p(n+2)}{p+n+2};\mathcal{O}_t}.\nonumber\\
\end{align*}
By assumption ($\mathcal{A}_2$), there exits $\theta>\varrho'$ such that $\lambda-\kappa-\theta\beta>0$ and $\frac{1}{\theta}+\frac{1}{\varrho}<1$. So, we can take $\epsilon$ and $\epsilon_1$ small enough such that $\lambda-\kappa-\beta\theta-\epsilon-\epsilon_1>0$ and $1-\frac{1}{\theta}-\frac{1}{\varrho}-\epsilon_1>0$. Taking the supremum on both sides, Lemma \ref{lemma2.6} yields,
\begin{align*}
&\|(u-k_m)^+\|^2_{\mathcal{V}_2(\mathcal{O}_t)}+\|v^{k_m}\|^2_{0,2;\mathcal{O}_t} \\
\leq & C_1(\lambda,\kappa,\beta,L,\theta,\varrho,\epsilon,\epsilon_1)\int_t^T\|(u-k_m)^+\|^2_{\mathcal{V}_2(\mathcal{O}_s)}\,ds\\
&+C_1(\lambda,\kappa,\beta,L,\theta,\varrho,n,\epsilon,\epsilon_1)\left(\frac{2^m}{k}\right)^{1+\frac{2(p-n-2)}{np}}\|(u-k_{m-1})^+\|^{2+\frac{2(p-n-2)}{np}}_{\mathcal{V}_2(\mathcal{O}_t)}\left(A_p(f_0^{+},g_0;\mathcal{O}_t)+Lk_m\right).
\end{align*}
Gronwall's inequality yields that
\begin{align}\label{gronwall}
&\|(u-k_m)^+\|_{\mathcal{V}_2(\mathcal{O}_t)}^2+\|v^{k_m}\|^2_{0,2;\mathcal{O}_t}\nonumber\\
\leq & C_2(\lambda,\kappa,\beta,L,\theta,\varrho,T,n,\epsilon,\epsilon_1)\left(\frac{2^m}{k}\right)^{1+\frac{2(p-n-2)}{np}}\|(u-k_{m-1})^+\|_{\mathcal{V}_2(\mathcal{O}_t)}^{2+\frac{2(p-n-2)}{np}}\left(A_p(f_0^+,g_0;\mathcal{O}_t)+Lk_m\right).
\end{align}
Letting $k\geq \bar{k}+\frac{A_p(f_0^+,g_0;\mathcal{O}_t)}{L}$, it follows from (\ref{gronwall}) that
\begin{align}\label{before-de-giorgi}
&\|(u-k_m)^+\|_{\mathcal{V}_2(\mathcal{O}_t)}^2+\|v^{k_m}\|^2_{0,2;\mathcal{O}_t}\nonumber\\
\leq & C_3(\lambda,\kappa,\beta,L,\theta,\varrho,T,n,\epsilon,\epsilon_1)\frac{2^{1+\frac{2(p-n-2)}{np}}}{k^{\frac{2(p-n-2)}{np}}}\left(2^{1+\frac{2(p-n-2)}{np}}\right)^{m-1}\|(u-k_{m-1})^+\|_{\mathcal{V}_2(\mathcal{O}_t)}^{2+\frac{2(p-n-2)}{np}}.
\end{align}
In terms of $a_m:=\|(u-k_m)^+\|_{\mathcal{V}_2(\mathcal{O}_t)}^2$, $C_0:=C_3(\lambda,\kappa,\beta,L,\theta,\varrho,T,\epsilon,\epsilon_1)\frac{2^{1+\frac{2(p-n-2)}{np}}}{k^{\frac{2(p-n-2)}{np}}}>0$, $b:=2^{1+\frac{2(p-n-2)}{np}}>1$ and $\delta:=\frac{(p-n-2)}{np}>0$, we get that
\begin{equation*} 
a_m\leq C_0b^{m-1}a_{m-1}^{1+\delta}.
\end{equation*}
Now, let
\[
	k\geq C_3(\lambda,\kappa,\beta,L,\theta,\varrho,T,\epsilon,\epsilon_1)^{\frac{1}{2\delta}}2^{(1+2\delta)\left(\frac{1}{2\delta^2}+\frac{1}{2\delta}\right)}\|({u}-\bar{k})^+\|_{\mathcal{V}_2(\mathcal{O}_t)}.
\]
	Then $a_0 \leq C_0^{-\frac{1}{\delta}}b^{-\frac{1}{\delta^2}}$. Therefore, Lemma \ref{lemma4.1} can be applied to get $\lim\limits_{m\rightarrow\infty}a_m=0$. Along with the above estimates for $k$ this implies that
\begin{align} \label{ubar-kbar}
\textrm{esssup}_{(\omega,t,x)\in\Omega\times \mathcal{O}_t}(u-\bar{k})^+\leq C\left(\bar{k}+A_p(f_0^+,g_0;\mathcal{O}_t)+\|(u-\bar{k})^+\|_{\mathcal{V}_2(\mathcal{O}_t)}\right),
\end{align}
where $C$ depends on $\lambda$, $\kappa$, $\beta$, $L$, $\varrho$, $T$, $p$ and $n$. The estimates of terms $\|(u-\bar{k})^+\|_{\mathcal{V}_2(\mathcal{O}_t)}$ and $\textrm{esssup}_{(\omega,t,x)\in\Omega\times\mathcal{O}_t}\hat{\xi}^+$ are given in the following Lemma \ref{estimate-u-k}(1) and Lemma \ref{MP-general-domain-quasilinear-BSPDE}(1).
Finally we arrive at
\begin{align}\label{final-step-MP}
&\textrm{esssup}_{(\omega,t,x)\in\Omega\times \mathcal{O}_t}u^+\nonumber\\
\leq & C \left(\textrm{esssup}_{(\omega,t,x)\in\Omega\times \partial_p\mathcal{O}_t}u^++\textrm{esssup}_{(\omega,t,x)\in\Omega \times \partial _p\mathcal{O}_t}{\hat{\xi}}^{+}\right.\nonumber\\
&\left.+A_p(f_0^+,g_0;\mathcal{O}_t)+B_2(f_0^+,g_0;\mathcal{O}_t)+A_p(\hat{f}^{+},\hat{g};\mathcal{O}_t)+B_2(\hat{f}^{+},\hat{g};\mathcal{O}_t)
\right),
\end{align}
where $C$ depends on $\lambda$, $\kappa$, $\beta$, $L$, $\varrho$, $T$, $p$ and $n$.

(2) For each $t\in[0,T]$, let $k\geq \textrm{esssup}_{(\omega,t,x)\in\Omega\times\partial_p\mathcal{O}_t}u^++\textrm{esssup}_{(\omega,t,x)\in\Omega\times\mathcal{O}_t}\hat{\xi}^+$. By Theorem \ref{Ito-u+} we obtain
\begin{align*} 
&\|(u-k)^+(t)\|^2+\int_t^T\|v^{k}(s)\|^2\,ds\nonumber\\
&=-2\int_t^T\langle\partial_j(u-k)^+(s),a^{ij}\partial_i(u-k)^+(s)+\sigma ^{jr}(s)v^{k,r}{s}\rangle\, ds\nonumber\\
&-2\int_t^T\langle\partial_j(u-k)^+(s),g^{j,k}(s,(u-k)^+(s),\nabla u(s),v^k(s))\rangle\, ds\nonumber\\
&+2\int_t^T\langle (u-k)^+(s),f^{k}(s,(u-k)^+(s),\nabla u(s),v^{k}(s))\rangle \,ds+2\int_{\mathcal{O}_t}(u-k)^+(s,x)\,\mu(ds,dx)\nonumber\\
&-2\int_t^T\langle (u-k)^+(s),v^{r,k}(s)\,dW^r_s\rangle.
\end{align*}
For every $k>l\geq \textrm{esssup}_{(\omega,t,x)\in\Omega\times\partial_p\mathcal{O}_t}u^++\textrm{esssup}_{(\omega,t,x)\in\Omega\times\mathcal{O}_t}\hat{\xi}^+$, we get
\begin{align*}
\left((u-l)^+-(u-k)^+\right)1_{(u>k)}=(k-l)1_{(u>k)},
\end{align*}
which implies
\begin{align*} 
1_{(u>k)}\leq \frac{(u-l)^+}{k-l}.
\end{align*}
By the assumptions on $g$ and $f$, and using the same arguments in (\ref{J_1}), (\ref{J_2-1}) and (\ref{J_2-2})-(\ref{gronwall}) we obtain that
\begin{align*}
\|(u-k)^+\|_{\mathcal{V}_2(\mathcal{O}_t)}^2+\|v^{k}\|^2_{0,2;\mathcal{O}_t}
\leq & C\frac{A_p(f_0^+,g_0;\mathcal{O}_t)}{(k-l)^{1+\frac{2(p-n-2)}{np}}}\|(u-l)^+\|_{\mathcal{V}_2(\mathcal{O}_t)}^{2+\frac{2(p-n-2)}{np}},
\end{align*}
where $C$ depends on $\lambda$, $\kappa$, $\beta$, $L$, $\varrho$, $T$ and $n$. By setting $\phi(k):=\|(u-k)^+\|_{\mathcal{V}_2(\mathcal{O}_t)}^2$, $\alpha:=1+\frac{2(p-n-2)}{np}>0$, $\zeta:=1+\frac{p-n-2}{np}$ and $C_1:=CA_p(f_0^+,g_0;\mathcal{O}_t)$, the following statement holds for each $k>l\geq \textrm{esssup}_{(\omega,t,x)\in\Omega\times\partial_p\mathcal{O}_t}u^++\textrm{esssup}_{(\omega,t,x)\in\Omega\times\mathcal{O}_t}\hat{\xi}^+$:
\begin{align*}
\phi(k)\leq \frac{C_1}{(k-l)^{\alpha}}\phi(l)^{\zeta}.
\end{align*}
If we define $d:=C_1^{\frac{1}{\alpha}}\left|\phi(\textrm{esssup}_{(\omega,t,x)\in\Omega\times\partial_p\mathcal{O}_t}u^++\textrm{esssup}_{(\omega,t,x)\in\Omega\times\mathcal{O}_t}\hat{\xi}^+)\right|^{\frac{\zeta-1}{\alpha}}2^{\frac{1+\alpha}{\alpha}}$, then by Corollary \ref{corollary-de-giorgi},
\begin{align*}
\|(u-d-\textrm{esssup}_{(\omega,t,x)\in\Omega\times\partial_p\mathcal{O}_t}u^+-\textrm{esssup}_{(\omega,t,x)\in\Omega\times\mathcal{O}_t}\hat{\xi}^+)^+\|_{\mathcal{V}_2(\mathcal{O}_t)}=0,
\end{align*}
and so Lemma \ref{estimate-u-k}(2) yields
\begin{align}\label{estimate-last-but-one-2}
\textrm{esssup}_{(\omega,t,x)\in\Omega\times \mathcal{O}_t}u^+\leq & \textrm{esssup}_{(\omega,t,x)\in\Omega\times\partial_p\mathcal{O}_t}u^++\textrm{esssup}_{(\omega,t,x)\in\Omega\times\mathcal{O}_t}\hat{\xi}^+\nonumber\\
&+CA_p(f_0^+,g_0;\mathcal{O}_t)^{\frac{1}{\alpha}}B_2(f_0^+,g_0;\mathcal{O}_t)^{\frac{2(\zeta-1)}{\alpha}},
\end{align}
where $C$ depends on $\lambda$, $\kappa$, $\beta$, $L$, $\varrho$, $n$, $p$ and $T$. Therefore
(\ref{estimate-last-but-one-2}) and (\ref{MP-general-domain-quasilinear-BSPDE-2}) yield
\begin{align*} 
\textrm{esssup}_{(\omega,t,x)\in\Omega\times \mathcal{O}_t}u^+\leq & \textrm{esssup}_{(\omega,t,x)\in\Omega\times\partial_p\mathcal{O}_t}u^++\textrm{esssup}_{(\omega,t,x)\in\Omega\times\partial_p\mathcal{O}_t}\hat{\xi}^+\nonumber\\
&+CA_p(\hat{f}^+,\hat{g};\mathcal{O}_t)^{\frac{1}{\alpha}}B_2(\hat{f}^+,\hat{g};\mathcal{O}_t)^{\frac{2(\zeta-1)}{\alpha}}\nonumber\\
&+CA_p(f_0^+,g_0;\mathcal{O}_t)^{\frac{1}{\alpha}}B_2(f_0^+,g_0;\mathcal{O}_t)^{\frac{2(\zeta-1)}{\alpha}}.
\end{align*}
\end{proof}

When the domain $\mathcal{O}$ is bounded, $\|\cdot\|_{0,2;Q}$ can be bounded by $\|\cdot\|_{0,p;Q}$ and $\|\cdot\|_{0,\frac{p(n+2)}{p+n+2};Q}$ and we have the following maximum principle for the RBSPDE (\ref{equation1.1}) on a bounded domain.

\begin{corollary}\label{MP-RBSPDE-bounded-domain}
\begin{itemize}
	\item[(1)] Assume $(\mathcal{A}_1)$-$(\mathcal{A}_5)$ hold and $\mathcal{O}$ is bounded. If the triplet $(u,v,\mu)$ is a solution to the RBSPDE (\ref{equation1.1}), then
\begin{align*} 
&\textrm{esssup}_{(\omega,t,x)\in\Omega \times Q}u^{\pm} \\
\leq& \, C\left(\textrm{esssup}_{(\omega,t,x)\in\Omega \times \partial _pQ}u^{\pm}+\textrm{esssup}_{(\omega,t,x)\in\Omega \times \partial _pQ}\hat{\xi}^{\pm}\right.\nonumber\\
&\quad \left.+A_p(f_0^{\pm},g_0;Q)+A_p(\hat{f}^{\pm},\hat{g};Q)\nonumber
 \right),
\end{align*}
where the constant $C$ depends only on $\lambda$, $\kappa$, $\beta$, $L$, $\varrho$, $T$, $p$, $n$ and $|\mathcal{O}|$.
	\item[(2)] Suppose that $(\mathcal{A}_1)$-$(\mathcal{A}_4)$ and (\ref{additional-assumption-f-g}) hold. Then for each solution $(u,v,\mu)$ to the RBSPDE (\ref{equation1.1}), it holds true that
\begin{align*} 
&\textrm{esssup}_{(\omega,t,x)\in\Omega\times Q}u^{\pm}\nonumber\\
\leq \,& \textrm{esssup}_{(\omega,t,x)\in\Omega\times\partial_pQ}u^{\pm}+\textrm{esssup}_{(\omega,t,x)\in\Omega\times\partial_pQ}\hat{\xi}^{\pm}\nonumber\\
&+C\left(A_p(f_0^{\pm},g_0;Q)+A_p(\hat{f}^{\pm},\hat{g};Q)\right),
\end{align*}
where $C$ depends on $\lambda$, $\kappa$, $\beta$, $L$, $\varrho$, $T$, $p$, $n$, and $|\mathcal{O}|$.
\end{itemize}
\end{corollary}

\begin{lemma}\label{estimate-u-k}
\begin{itemize}
	\item[(1)] Under the same conditions as in Theorem \ref{Theorem_MP_GD}(1), for each $t\in[0,T]$ and each $k\geq \textrm{esssup}_{(\omega,t,x)\in\Omega\times\partial_p\mathcal{O}_t}u^++\textrm{esssup}_{(\omega,t,x)\in\Omega\times\mathcal{O}_t}\hat{\xi}^+$, we have
\begin{align*}
\|(u-k)^{+}\|_{\mathcal{V}_2(\mathcal{O}_t)}\leq C(B_2(f_0^{+},g_0;\mathcal{O}_t)+k),
\end{align*}
where $C$ depends on $\lambda$, $\kappa$, $\beta$, $L$, $\varrho$ and $T$.
	\item[(2)] Under the same conditions as in Theorem \ref{Theorem_MP_GD}(2), for each $t\in[0,T]$ and $k\geq \textrm{esssup}_{(\omega,t,x)\in\Omega\times\partial_p\mathcal{O}_t}u^++\textrm{esssup}_{(\omega,t,x)\in\Omega\times\mathcal{O}_t}\hat{\xi}^+$ and we have
\begin{align*}
\|(u-k)^+\|_{\mathcal{V}_2(\mathcal{O}_t)}\leq CB_2(f_0^+,g_0;\mathcal{O}_t),
\end{align*}
where $C$ depends on $\lambda$, $\kappa$, $\beta$, $L$, $\varrho$ and $T$.
\end{itemize}
\end{lemma}
\begin{proof}
\begin{itemize}
\item[(1)] As in the proof of Theorem \ref{Theorem_MP_GD}, we may assume w.l.o.g that $f(t,x,r,0,0)$ is non-increasing in $r$. For
\[	
    k\geq\textrm{esssup}_{(\omega,t,x)\in\Omega\times\partial_{p}\mathcal{O}_t}u^++\textrm{esssup}_{(\omega,t,x)\in\Omega\times\mathcal{O}_t}\hat{\xi}^+,
\]
we have
\begin{align*}
\int_t^T\int_{\mathcal{O}}(u-k)^+\,\mu(dxds) \leq \int_t^T\int_{\mathcal{O}}(u-\xi)^+\,\mu(dxds)+\int_t^T\int_{\mathcal{O}}(\xi-\hat{\xi}^+)^+\,\mu(dxds)=0.
\end{align*}
\end{itemize}
Applying Theorem \ref{Ito-u+}, we have
\begin{equation}
\begin{split}
&\|(u-k)^+(t)\|^2+E\left(\int_t^T\|v^{k}(s)\|^2\,ds|\mathcal{F}_t\right)
\\
\leq &-2E\left(\int_t^T\langle\partial_j(u-k)^+(s),a^{ij}\partial_i(u-k)^+(s)+\sigma ^{jr}(s)v^{k,r}(s)\rangle\, ds|\mathcal{F}_t\right)\\
&-2E\left(\int_t^T\langle\partial_j(u-k)^+(s),g^{j,k}\left(s,(u-k)^+(s),\nabla u(s),v^{k}(s)\right)\rangle\, ds|\mathcal{F}_t\right)\\
&+2E\left(\int_t^T\langle (u-k)^+(s),f^{k}\left(s,(u-k)^+(s),\nabla u(s),v^{k}(s)\right)\rangle \,ds|\mathcal{F}_t\right)\\
:=& K_1+K_2+K_3,
\end{split}
\end{equation}
where $v^{r,k}:=v^r1_{\{u>k\}}$, $g^{j,k}(\cdot,\cdot,\cdot,X,\cdot,\cdot):=g^j(\cdot,\cdot,\cdot,X+k,\cdot,\cdot)$, $f^{k}(\cdot,\cdot,\cdot,X,\cdot,\cdot):=f(\cdot,\cdot,\cdot,X+k,\cdot,\cdot)$. The quantities $K_i$ $(i=1,2,3)$ can now be estimated by analogy to the constants $I_i$ $(i=1,2,3)$ in the proof of Theorem \ref{theorem3.2}. Specifically, $K_1$ can be estimated as $I_1$, with $u-\hat{\xi}$ and $v-\hat{v}$ being replaced by $(u-k)^+$ and $v$, respectively; $K_2$ can be estimated as $I_2$, without $\hat{g}$ (because we now have no obstacle process involved in), and $u-\hat{\xi}$ and $g^j(s,x,u,\nabla u,v)$ being replaced by $(u-k)^+$ and $g^{j,k}\left(s,x,(u-k)^+(s),\nabla u(s),v^{k}(s)\right)$, respectively and the estimate for $K_3$ is similar to that for $I_3$, without $\hat{f}$ and $u-\hat{\xi}$ and $f(s,x,u,\nabla u,v)$ being replaced by $(u-k)^+$ and $f^{k}\left(s,x,(u-k)^+,\nabla u,v^{k}\right)$, respectively. Finally, by $(\mathcal{A}_5)$, $\|g^k_0\|_{0,2;\mathcal{O}_t}$ can be estimated by $\|g_0\|_{0,2;\mathcal{O}_t}+Lk$. This yields the desired result.
\item[(2)] The proof is the same as that of $(1)$ if we note that $\|g^k_0\|_{0,2;\mathcal{O}_t}=\|g_0\|_{0,2;\mathcal{O}_t}$ by assumption.
\end{proof}

The following lemma establishes the maximum principle for quasi-linear BSPDE on general domains.
\begin{lemma}\label{MP-general-domain-quasilinear-BSPDE}
Let $(u,v)$ be a weak solution to the following quasi-linear BSPDE
\begin{equation}\label{quasilinear-BSPDE}
\left\{\begin{array}{ll}
\begin{split}
-du(t,x)&=[\partial _j(a^{ij}\partial _iu(t,x)+\sigma ^{jr}v^r(t,x))
+f(t,x,u(t,x),\nabla u(t,x),v(t,x))\\
&+\nabla \cdot g(t,x,u(t,x),\nabla u(t,x),v(t,x))]\,dt
-v^r(t,x)\,dW^r_t,
\quad (t,x)\in Q,\\
u(T,x)&=G(x),~~~x\in \mathcal{O}.
\end{split}
\end{array}
\right.
\end{equation}
\begin{itemize}
	\item[(1)] If the coefficients satisfy assumptions $(\mathcal{A}_1)$, $(\mathcal{A}_2)$, $(\mathcal{A}_3)$ and $(\mathcal{A}_5)$, then for each $t\in[0,T]$ we have
\begin{equation}\label{MP-general-domain-quasilinear-BSPDE-1}
\begin{split}
&\textrm{esssup}_{(\omega,t,x)\in\Omega \times \mathcal{O}_t}u^{\pm} \\
\leq \, & C\left(\textrm{esssup}_{(\omega,t,x)\in\Omega \times \partial _p\mathcal{O}_t}u^{\pm}+A_p(f_0^{\pm},g_0;\mathcal{O}_t)+B_2(f_0^{\pm},g_0;\mathcal{O}_t)\right)
\end{split}
\end{equation}
where $C$ depends on $\lambda$, $\kappa$, $\beta$, $L$, $\varrho$, $T$, $p$ and $n$;
	\item[(2)] If $(\mathcal{A}_1)$, $(\mathcal{A}_2)$, $(\mathcal{A}_3)$ and (\ref{additional-assumption-f-g}) hold true,
then for each $t\in[0,T]$ we have
\begin{align}\label{MP-general-domain-quasilinear-BSPDE-2}
&\textrm{esssup}_{(\omega,t,x)\in\Omega\times \mathcal{O}_t}u^{\pm}\nonumber\\
\leq \, & \textrm{esssup}_{(\omega,t,x)\in\Omega\times\partial_p\mathcal{O}_t}u^{\pm}\nonumber\\
&+C A_p(f_0^{\pm},g_0;\mathcal{O}_t)^{\frac{np}{np+2(p-n-2)}}B_2({f}_0^{\pm},{g}_0;\mathcal{O}_t)^{\frac{2(p-n-2)}{np+2(p-n-2)}},
\end{align}
where $C$ depends on $\lambda$, $\kappa$, $\beta$, $L$, $\varrho$, $n$, $p$ and T.
\end{itemize}
\end{lemma}

\begin{proof}
    In terms of $\bar{k}=\textrm{esssup}_{(\omega,t,x)\in\Omega\times\partial_{p}\mathcal{O}_t}u^+$ the assertion follows by establishing estimates analogous to (\ref{J_1})-(\ref{ubar-kbar}) and Lemma \ref{estimate-u-k}(1).
\end{proof}

The proceeding lemmas allow us to establish the comparison principle for the quasi-linear BSPDE on a general domain.
\begin{corollary}\label{comparison-quasi-linear-BSPDE}
Let $(u_i,v_i)$ be solutions to the quasi-linear BSPDE (\ref{quasilinear-BSPDE}) with parameters $(f_i,g,G_i,a,\sigma)$ respectively, $i=1,2$. Suppose that assumptions in Lemma 4.4 hold and that  $(u_1-u_2)^+|_{\partial\mathcal{O}}=0$. Then if $f_1(t,x,u_2,\nabla u_2,v_2)\leq f_2(t,x,u_2,\nabla u_2,v_2)$ $dt\times dx\times d\mathbb{P}\textrm{-}a.e.$ and $G_1\leq G_2$ $dx\times d\mathbb{P}\textrm{-}a.e.$, we have $u_1\leq u_2$ $dt\times dx\times d\mathbb{P}\textrm{-}a.e.$.
\begin{proof}
Let $(\underline{u},\underline{v})=(u_1-u_2,v_1-v_2)$. Then $(\underline{u},\underline{v})$ is a solution to the quasi-linear BSPDE (\ref{quasilinear-BSPDE}) with parameters $(\underline{f},\underline{g},\underline{G},a,\sigma)$, where
\begin{align*}
&\underline{f}(t,x,\cdot,\cdot,\cdot)=f_1(t,x,\cdot+u_2,\cdot+\nabla u_2,\cdot+v_2)-f_2(t,x,u_2,\nabla u_2,v_2)\nonumber\\
&\underline{g}(t,x,\cdot,\cdot,\cdot)=g(t,x,\cdot+u_2,\cdot+\nabla u_2,\cdot+v_2)-g(t,x,u_2,\nabla u_2,v_2)\nonumber\\
&\underline{G}=G_1-G_2.\nonumber
\end{align*}
Then we have $\underline{f}_0:=\underline{f}(\cdot,\cdot,0,0,0)\leq 0$, $\underline{g}_0:=\underline{g}(\cdot,\cdot,0,0,0)=0$ and $\textrm{esssup}_{\Omega\times\partial_p Q}u^+=0$. Therefore by Lemma \ref{estimate-u-k} or Lemma \ref{MP-general-domain-quasilinear-BSPDE}, there holds that $u_1\leq u_2$ $dt\times dx\times d\mathbb{P}\textrm{-}a.e.$.
\end{proof}
\end{corollary}

\subsection{Local Behavior of the Random Field $u^{\pm}$}
The global maximum principle in Theorem \ref{Theorem_MP_GD} tells us that if the random field $u^{\pm}$ is bounded on the parabolic boundary, it must be bounded in the whole domain. This section studies the local behavior of $u^{\pm}$ when it is not necessarily bounded on the parabolic boundary.
\begin{definition}\label{cutofffunction}
A function $\zeta$ is called a cut-off function on the sub-domain $Q'\subset Q$ if it satisfies the following properties:
\begin{itemize}
	\item[(1)] there exits some smooth function sequence $\{\zeta_m\}\subset C_0^{\infty}(Q')$ such that $\zeta_m$, $\partial_s\zeta_m$ and $\nabla\zeta_m$ converge to $\zeta$, $\partial_s\zeta$ and $\nabla\zeta$ in $L^{\infty}(Q')$ respectively;
	\item[(2)] $\zeta\in [0,1]$;
	\item[(3)] there exits a domain $Q'' \subset\subset Q'$ and a nonempty domain $Q''' \subset\subset Q''$ such that
            $$  \zeta(t,x)=
   \begin{cases}
   0 &\mbox{if $(t,x)\in Q' \backslash Q''$}\\
   1 &\mbox{if $(t,x)\in Q'''$,}
   \end{cases}$$
   where by $A\subset\subset B$ we mean the closure $\bar{A}\subseteq B$.
\end{itemize}
\end{definition}

We modify the definition of backward stochastic parabolic De Giorgi class in \cite{qiutang} as follows.
\begin{definition}
We say a function $u\in\mathcal{V}_{2,0}(Q)$ belongs to a backward stochastic parabolic De Giorgi class $BSPDG^{\pm}(a_0,b_0,k_0,\eta;\delta,Q)$ with $$(a_0,b_0,k_0,\eta,\delta)\in[0,\infty)\times[0,\infty)\times[0,\infty)\times(n+2,\infty)\times(0,1),$$
 if for any $Q_{\rho,\tau}:=[t_0-\tau,t_0)\times B_{\rho}(x_0)\subset Q$ with $(\rho,\tau)\in (0,\delta]\times (0,\delta^2]$, each cut-off function $\zeta$ on $Q_{\rho,\tau}$ and for each $k\geq k_0$, we have
\begin{align}
\|\zeta (u-k)^{\pm}\|^2_{\mathcal{V}_2(Q_{\rho,\tau})}
&\leq b_0 \bigg\{
\|(u-k)^{\pm}\|^2_{0,2;Q_{\rho,\tau}}\left(1+\|\partial _t\zeta\|_{L^{\infty}(Q_{\rho,\tau})}+\|\nabla \zeta\|^2_{L^{\infty}(Q_{\rho,\tau})}\right)\nonumber\\
&\quad \left.+\left(k^2+a_0^2\right)|{(u-k)^{\pm}>0}|_{\infty;Q_{\rho,\tau}}^{1-\frac{2}{\eta}}\right\},
\label{relatn-DG}
\end{align}
where $|{(u-k)^{\pm}>0}|_{\infty;Q_{\rho,\tau}}:=\textrm{esssup}_{\omega\in \Omega}\sup_{s\in[t_0-\tau,t_0)}E\left[\int_{[s,t_0)\times B_{\rho}(x_0)}1_{\{(u(t,x)-k)^{\pm}>0\}}\,dxdt|\mathcal{F}_s\right]$.
\end{definition}

Here, we take $(k,\rho, \tau)\in [k_0,\infty)\times (0,\delta]\times (0,\delta^2]$ for given $(k_0,\delta)\in [0,\infty)\times(0,1)$ in the above definition, instead of $(k,\rho, \tau)\in \bR\times (0,1)\times (0,1)$ as in \cite[Definition 5.2]{qiutang}. However, a direct extension of \cite[Theorem 5.8]{qiutang} yields the following lemma.
%
\begin{lemma}\label{lem-DG}
Given $k_0^{\pm}\geq 0$, if $u\in BSPDG^{\pm}(a^{\pm}_0,b^{\pm}_0,k_0^{\pm},\eta;\delta,Q)$, then
\begin{align*} 
\textrm{esssup}_{(\omega,t,x)\in \Omega\times Q_{\frac{\rho}{2}}}u^{\pm}
&\leq 2k_0^{\pm} + C_{\pm} \left\{\rho^{-\frac{n+2}{2}}\|u^{\pm}\|_{0,2;Q_{\rho}}+a_0^{\pm}\rho^{1-\frac{2+n}{\eta}}\right\},
\end{align*}
where $Q_{\rho}:=[t_0-\rho ^2,t_0)\times B_{\rho}(x_0)\subset Q$ with $\rho \in (0,\delta]$ and the constants $C_{\pm}$ depend on $a^{\pm}_0$, $b^{\pm}_0$ and $n$.
\end{lemma}

%

For the solution to the RBSPDE (\ref{equation1.1}), we further have the following result.

\begin{lemma}\label{RBSPDE-DG}
Let assumptions $(\mathcal{A}_1)$-$(\mathcal{A}_4)$ hold. Suppose $(u,v,\mu)$ is a solution to the RBSPDE (1.1). Given $Q_{\delta}:=[t_0-\delta ^2,t_0)\times B_{\delta}(x_0)\subset Q$ with $\delta \in (0,1)$, let $k_0^{\pm}=\textrm{esssup}_{\Omega\times Q_{\delta}}\hat{\xi}^{\pm}$. Then we have $u\in BSPDG^{\pm}(a^{\pm}_0,b_0,k^{\pm}_0,\eta;\delta,Q_{\delta})$ with $\eta=p$, $a^{\pm}_0=A_p(f_0^{\pm},g_0;Q_{\delta})$ and $b_0$ depending on $\lambda$, $\kappa$, $\beta$, $\varrho$, $\Lambda$, $L$, $n$ and $p$.
\end{lemma}

\begin{proof}
First we generalize the It\^o formula to a local case for the RBSPDE (\ref{equation1.1}).
For each cut-off function $\zeta$ on $Q_{\rho,\tau}$ with $(\rho,\tau)\in(0,\delta]\times(0,\delta^2]$, we can choose a sequence of smooth functions $\{\zeta _m\}\subset C_0^{\infty}(Q_{\rho,\tau})$ such that $\zeta_m$ and its gradients w.r.t. $s$ and $x$ converge uniformly to $\zeta$ and its gradient, respectively, as $m\rightarrow \infty$.

For $k \geq k_0^+$, Theorem \ref{Ito-u+} yields that
\begin{align*}
&\|(u-k)^+(t)\zeta_m(t)\|^2_{L^2(B_{\rho}(x_0))}+\int_t^{t_0}\|\zeta_m(s)v^k(s)\|^2_{L^2(B_{\rho}(x_0))}\,ds\\
=&-2\int_t^{t_0}\langle\zeta_m(s)\partial _s\zeta_m(s),|(u-k)^+(s)|^2\rangle_{B_{\rho}(x_0)}\,ds
+2\int_t^{t_0}\langle\zeta_m^2(s)(u-k)^+(s),f(s,u(s),\nabla u(s),v(s))\rangle_{B_{\rho}(x_0)}\,ds\\
&-2\int_t^{t_0}\langle\partial _j(\zeta_m^2(s)(u-k)^+(s)),a^{ij}(s)\partial_iu(s)+\sigma ^{jr}(s)v^{r}(s)+g^{j}(s,u(s),\nabla u(s),v(s))\rangle_{B_{\rho}(x_0)}\,ds\\
&-2\int_t^{t_0}\langle\zeta_m^2(s)(u-k)^+(s),v^{r,k}(s)\rangle_{B_{\rho}(x_0)}\,dW^r_s+2\int_t^{t_0}\int_{B_{\rho}(x_0)}(u-k)^+(s,x)\zeta^2_m\,\mu(ds,dx),
\end{align*}
where $v^{r,k}:=v^r1_{\{u>k\}}$.

Thus, by letting $m\rightarrow \infty$ and the dominated convergence theorem, we can get
\begin{align*}
&\|(u-k)^+(t)\zeta(t)\|^2_{L^2(B_{\rho}(x_0))}+\int_t^{t_0}\|\zeta(s)v^k(s)\|^2_{L^2(B_{\rho}(x_0))}\,ds\nonumber\\
&=-2\int_t^{t_0}\langle\zeta(s)\partial _s\zeta(s),|(u-k)^+(s)|^2\rangle_{B_{\rho}(x_0)}\,ds
+2\int_t^{t_0}\langle\zeta^2(s)(u-k)^+(s),f(s,u(s),\nabla u(s),v(s))\rangle_{B_{\rho}(x_0)}ds\nonumber\\
&-2\int_t^{t_0}\langle\partial _j(\zeta^2(s)(u-k)^+(s)),a^{ij}(s)\partial_iu(s)+\sigma ^{jr}(s)v^{r}(s)
+g^{j}(s,u(s),\nabla u(s),v(s))\rangle_{B_{\rho}(x_0)}\,ds\nonumber\\
&-2\int_t^{t_0}\langle\zeta^2(s)(u-k)^+(s),v^{r,k}(s)\rangle_{B_{\rho}(x_0)}\,dW^r_s
+2\int_t^{t_0}\int_{B_{\rho}(x_0)}(u-k)^+(s,x)\zeta^2\,\mu(ds,dx).
\end{align*}
Taking conditional expectation, we obtain
\begin{align}\label{startingpoint_local}
&\|((u-k)\zeta)^+(t)\|^2_{L^2(B_{\rho}(x_0))}+E\left[\int_t^{t_0}\|\zeta(s)v^k(s)\|^2_{L^2(B_{\rho}(x_0))}\,ds|\mathcal{F}_t\right]\nonumber\\
&=-2E\left[\int_t^{t_0}\langle\zeta(s)\partial _s\zeta(s),|(u-k)^+(s)|^2\rangle_{B_{\rho}(x_0)}\,ds|\mathcal{F}_t\right]\nonumber\\
&\quad+2E\left[\int_t^{t_0}\langle\zeta^2(s)(u-k)^+(s),f^k(s,(u(s)-k)^+,\nabla (u(s)-k)^+,v(s))\rangle_{B_{\rho}(x_0)}\,ds|\mathcal{F}_t\right]\nonumber\\
&\quad-2E\left[\int_t^{t_0}\langle\partial _j(\zeta^2(s)(u-k)^+(s)),a^{ji}(s)\partial_iu(s)+\sigma ^{jr}(s)v^{r}(s)\right.\nonumber\\
&\quad\quad\quad\quad\quad\quad +g^{j,k}(s,(u(s)-k)^+,\nabla (u(s)-k)^+,v(s))\rangle_{B_{\rho}(x_0)}\,ds|\mathcal{F}_t\bigg]\nonumber\\
&\quad +2E\left[\int_t^{t_0}\int_{B_{\rho}(x_0)}(u-k)^+(s,x)\zeta^2\,\mu(ds,dx)|\mathcal{F}_t\right],
\end{align}
where $f^k(\cdot,\cdot,\cdot,X,Y,Z):=f(\cdot,\cdot,\cdot,X+k,Y,Z)$ and $g^{j,k}(\cdot,\cdot,\cdot,X,Y,Z):=g^j(\cdot,\cdot,\cdot,X+k,Y,Z)$. As $k \geq k_0^+$, the last term on the right hand side of \eqref{startingpoint_local} vanishes. Hence, starting from (\ref{startingpoint_local}), we derive the desired result in a similar way to \cite[Proposition 5.6]{qiutang}.
\end{proof}


Given $Q_{2\rho}:=[t_0-4\rho ^2,t_0)\times B_{2\rho}(x_0)\subset Q$ with $\rho \in (0,1)$,
 let $k_0^{\pm}=\textrm{esssup}_{\Omega\times Q_{\rho}}\hat{\xi}^{\pm}$. Lemma \ref{RBSPDE-DG} shows that $u\in BSPDG^{\pm}(a^{\pm}_0,b_0,k_0^{\pm},\eta;\rho,Q_{\rho})$ with $\eta=p$. $a^{\pm}_0=A_p(f_0^{\pm},g_0;Q_{\rho})$ and $b_0$ given therein. On the other hand,
in view of the local boundedness of weak solutions for BSPDEs (\cite[Proposition 5.6 and Theorem 5.8]{qiutang}), we have
\begin{align*}
k_0^{\pm}
&\leq C \left\{\rho^{-\frac{n+2}{2}}\|\hat{\xi}^{\pm}\|_{0,2;Q_{2\rho}}+A_p(\hat{f}^{\pm},\hat{g};Q_{2\rho})\rho ^{1-\frac{2+n}{p}}\right\}
\end{align*}
with $C$ depending on $\lambda$, $\kappa$, $\varrho$, $\Lambda$, $n$ and $p$. Hence, further by Lemmas \ref{lem-DG} and \ref{RBSPDE-DG}, we obtain finally the local behavior of weak solutions to the RBSPDE (1.1).

\begin{theorem}\label{MP-RBSPDE-Local}
Let assumptions $(\mathcal{A}_1)$-$(\mathcal{A}_4)$ hold. Let $(u,v,\mu)$ be a weak solution to the RBSPDE (1.1). Given $Q_{2\rho}:=[t_0-4\rho ^2,t_0)\times B_{2\rho}(x_0)\subset Q$ with $\rho \in (0,1)$, we have
\begin{equation*}
\begin{split}
\textrm{esssup}_{(\omega,s,x)\in \Omega\times Q_{\frac{\rho}{2}}}u^{\pm}
\leq &\, C \left\{\rho^{-\frac{n+2}{2}}(\|u^{\pm}\|_{0,2;Q_{\rho}}+\|\hat{\xi}^{\pm}\|_{0,2;Q_{2\rho}})\right.\\
&\left.+\left(A_p(f_0^{\pm},g_0;Q_{\rho})+A_p(\hat{f}^{\pm},\hat{g};Q_{2\rho})\right)\rho ^{1-\frac{2+n}{p}}\right\},
\end{split}
\end{equation*}
where $C$ is a positive constant depending on $\lambda$, $\kappa$, $\beta$, $\varrho$, $\Lambda$, $L$, $n$ and $p$.
\end{theorem}

\begin{appendix}
\section{Auxiliary lemmas and It\^o's Formulas}

This subsection states some useful lemmas and It\^o formulas, which have been frequently used. The first lemma and corollary are from \cite{chen}.

\begin{lemma}\label{lemma4.1}
Let $\{a_k,k\in\mathbb{N}\}$ be a sequence of nonnegative numbers satisfying
$$a_{k+1}\leq C_0b^ka_k^{1+\delta},$$
where $b>1$, $\delta >0$ and $C_0$ is a positive constant. Then, if $a_0\leq \theta _0:=C_0^{-\frac{1}{\delta}}b^{-\frac{1}{\delta^2}}$, we have $\lim_{k\rightarrow \infty}a_k=0$.
\end{lemma}

\begin{corollary}\label{corollary-de-giorgi}
Let $\phi :[r_0,\infty]\rightarrow\mathbb{R}^+$ be a nonnegative and decreasing function. Assume there exist constants $C_1>0$, $\alpha >0$ and $\varsigma >1$ such that for any $r_0<r<l$,
$$\phi(l)\leq \frac{C_1}{(l-r)^{\alpha}}\phi (r)^{\varsigma}.$$
Then for any $d$ satisfying
$$d\geq C_1^{\frac{1}{\alpha}}|\phi(r_0)|^{\frac{\varsigma-1}{\alpha}}2^{\frac{\varsigma}{\varsigma-1}},$$
we have $\phi(r_0+d)=0$.
\end{corollary}

The following embedding lemma is from \cite{qiutang}.
\begin{lemma}\label{lemma2.6}
If for each $t\in[0,T]$, $u\in\mathcal{V}_2(\mathcal{O}_t)$, then we have
$$\|u\|_{0,\frac{2(n+2)}{n};\mathcal{O}_t}\leq C\|\nabla u\|^{\frac{n}{n+2}}_{0,2;\mathcal{O}_t}\textrm{esssup}_{(\omega,s)\in\Omega\times[t,T]}\|u(\omega,s)\|^{\frac{n}{n+2}}\leq C\|u\|_{\mathcal{V}_2(\mathcal{O}_t)},$$
where $C$ only depends on $n$.
\end{lemma}

 Now, we are going to present the It\^o formulas, which have been frequently used in the main text. We assume that $\Phi$ is a function that satisfies the following properties:
 \begin{itemize}
\item[(1)] $\Phi\in \mathcal{C}(\mathbb{R}^+\times\mathbb{R}^{n}\times\mathbb{R}\rightarrow \mathbb{R})$ and $\partial_t\Phi(t,x,u)$, $\Phi'(t,x,u)$, $\Phi''(t,x,u)$ and $\partial_j\Phi'(t,x,u)$, $j=1,2,\cdots,n$ exist and are continuous;
\item[(2)] $\Phi'(t,x,0)=0$ for any $(t,x)\in \mathbb{R}^+\times\mathbb{R}^n$;
\item[(3)] $\sup_{t\in\mathbb{R}^{+},x\in\mathbb{R}^n}|\partial_j\Phi'(t,x,u)|\leq C |u|$, $j=1,2,\cdots,n$;
\item[(4)]
\begin{align*}
\sup\limits_{t\in\mathbb{R}^{+},x\in\mathbb{R}^n,u\in\mathbb{R}/\{0\}} \left\{|\Phi''(t,x,u)|+\frac{1}{|u|^2}|\partial_t\Phi(t,x,u)-\partial_t\Phi(t,x,0)|\right\}<\infty,
\end{align*}
where $\partial_j\Phi(t,x,u)=\partial_{x_j}\Phi(t,x,u)$, $\Phi'(t,x,u)=\partial_u\Phi(t,x,u)$ and $\Phi''(t,x,u)=\partial^2_u\Phi(t,x,u)$.
\end{itemize}
Suppose that the following BSPDE
\begin{equation}\label{appeneq}
\left\{\begin{array}{ll}
\begin{split}
-du(t,x)&=[\partial _j(a^{ij}\partial _iu(t,x)+\sigma ^{jr}v^r(t,x))
+\bar{f}(t,x)+\nabla \cdot \bar{g}(t,x)]\,dt \\
& +\mu (dt,x)-v^r(t,x)\,dW^r_t,
~~~(t,x)\in Q,\\
u(T,x)&=G(x),~~~x\in \mathcal{O},\\
\end{split}
\end{array}
\right.
\end{equation}
holds in the weak sense where $(u,v)\in\mathcal{V}_2(Q)\times\mathcal{M}^{0,2}(Q)$, $\mu$ is a stochastic regular measure, $\bar{f}$, $\bar{g}$ and $G$ satisfy $(\mathcal{A}_3)$, $a$ and $\sigma$ satisfy $(\mathcal{A}_2)$.

When $\Phi$ is independent of $x$, i.e., $\Phi(t,x,u)=\Phi(t,u)$, the first It\^o formula is from \cite[Theorem 3.10]{qiuwei}.

\begin{proposition}\label{proposition6.1}
Let BSPDE (\ref{appeneq}) hold in the weak sense with $u|_{\partial \mathcal{O}}=0$. Then there holds almost surely that
\begin{align*}
&\int_{\mathcal{O}}\Phi(t,u(t,x))\,dx+\frac{1}{2}\int_t^T\langle\Phi''(s,u(s)),|v(s)|^2\rangle \,ds\\
=& \int_{\mathcal{O}}\Phi(T,G(x))\,dx-\int_t^T\int_{\mathcal{O}}\partial_s\Phi(s,u(s,x))\,dxds+\int_t^T\langle \Phi'(s,u(s)),\bar{f}(s)\rangle \,ds\\
&-\int_t^T\langle \Phi''(s,u(s))\partial_{j}u(s),a^{ij}(s)\partial_{i}u(s)+\sigma^{jr}(s)v^r(s)+\bar{g}^j(s)\rangle \,ds\\
&+\int_{[t,T]\times\mathcal{O}}\Phi'(s,u(s,x))\,\mu(ds,dx)-\int_t^T\langle\Phi'(s,u(s)),v^r(s)\rangle \,dW_s^r.
\end{align*}
\end{proposition}

The following It\^o formula extends the preceding one to the positive parts of the weak solutions to BSPDEs.

\begin{theorem}\label{Ito-u+}
Let BSPDE (\ref{appeneq}) hold in the weak sense but with $u^+|_{\partial \mathcal{O}}=0$. Then there holds almost surely that
\begin{equation}\label{equation6.5}
\begin{split}
        &\int_{\mathcal{O}}\Phi(t,x,u^+(t,x))\,dx+\frac{1}{2}\int_t^T\langle \Phi''(s,u^+(s)),|v^u(s)|^2\rangle \,ds\\
        =& \int_{\mathcal{O}}\Phi(T,x,G^+(x))\,dx-\int_t^T\int_{\mathcal{O}}\partial_s\Phi(s,x,u^+(s,x))\,dxds\\
        &+\int_t^T\langle \Phi'(s,u^+(s)),\bar{f}^u(s)\rangle\,ds+\int_t^T\int_{\mathcal{O}}\Phi'(s,x,u^+(s,x))\,\mu(dsdx)\\
        &-\int_t^T\langle \Phi''(s,u^+(s))\partial_j u^+(s)+\partial_j \Phi'(s,u^+(s)),a^{ij}(s)\partial_i u^+(s)+\sigma^{j,r}(s)v^{r,u}(s)+\bar{g}^{j,u}(s)\rangle \,ds\\
        &-\int_t^T\langle\Phi'(s,u^+(s)),v^{r,u}(s)\rangle \,dW^r_s,
\end{split}
\end{equation}
where
$$v^{r,u}=1_{\{u>0\}}v^r,~~~~\bar{f}^u=1_{\{u>0\}}\bar{f},~~~~\bar{g}^{j,u}=1_{\{u>0\}}\bar{g}^j.$$
\end{theorem}

\begin{proof}
Note that in general we cannot get $u|_{\partial\mathcal{O}}=0$ from $u^+|_{\partial\mathcal{O}}=0$, so Proposition \ref{proposition6.1} is not applicable. Insread, we shall apply an approximation scheme similar to that for \cite[Theorem 3.10]{qiuwei}

Let $\check{u}$ be the stochastic regular parabolic potential (see next subsection for the definition) associated with $\mu$. Now define
\begin{equation*}
\left\{\begin{array}{ll}
\begin{split}
-d\hat{u}(t,x)&=\left(-\Delta \hat{u}(t,x)+\bar{f}(t,x)+\nabla\cdot\hat{g}(t,x)\right)dt
-v^r(t,x)dW^r_t,\\
&~~~(t,x)\in Q,\\
\hat{u}(0,x)&=u(0,x),~~~x\in \mathcal{O},\\
\end{split}
\end{array}
\right.
\end{equation*}
where $\hat{g}^j(t,x)=\partial_j u(t,x)+a^{ij}\partial _iu(t,x)+\sigma ^{jr}v^r(t,x)+\bar{g}^j(t,x)$. Then, $u=\hat{u}-\check{u}$ and the zero Dirichlet conditions of $u^+$ and $\check{u}$ imply $\hat{u}^+|_{\partial\mathcal{O}}=0$. By \cite[Proposition 3.9(i)]{qiuwei} $u$ is almost surely quasi-continuous. So, the integral w.r.t. $\mu$ in (\ref{equation6.5}) is well defined. We can also check that all the other terms in (\ref{equation6.5}) are well defined.

  Thus, by Proposition 3.9(iv) and Remark 3.7 in \cite{qiuwei}, there exist $f^n\in \mathcal{L}^2([0,T];(H^{-1})^+(\mathcal{O}))$, $\check{v}^n\in \mathcal{L}^2([0,T];(L^2(\mathcal{O}))^m)$, $\check{u}^n\in \mathcal{U}(-\infty,f^{n}_1,g^{n}_1,G^n_1)$ and $\phi^n\in\mathcal{U}(-\infty,f^{n}_2,g^{n}_2,G^n_2)$, for some $f^{n}_i\in\mathcal{L}^2([0,T];L^2(\mathcal{O}))$, $g^{n}_i\in\mathcal{L}^2([0,T];(L^2(\mathcal{O}))^n)$, $G^n_i\in\mathcal{L}^2(\mathcal{O})$, $i=1,2$, such that $\phi^n\downarrow 0$ as $n\rightarrow\infty$, $dt\times dx\times d\mathbb{P}$ a.e., $\lim_{n\rightarrow\infty}\sum_{i=1}^m E\int_0^T\|\check{v}^{n,i}(t)\|^2\,dt=0$, $\lim_{n\rightarrow\infty}\|\check{u}^n-\check{u}\|_{\mathcal{L}^2(\mathcal{K})}=0$, $\lim_{n\rightarrow\infty}(\|f^n_2+\nabla\cdot g^n_2\|_{\mathcal{L}^2([0,T];H^{-1}(\mathcal{O}))}+\|G_2^n\|_{\mathcal{L}^2(\mathcal{O})})=0$, $|\check{u}^n-\check{u}|\leq \phi^n$ $dt\times dx\times d\mathbb{P}$ a.e., with $\check{u}^n$ satisfying the SPDE
    \begin{equation*}
    \begin{split}
    \left\{\begin{array}{ll}
      d\check{u}^n(t,x) = & ~ [\Delta \check{u}^n(t,x)+f^n(t,x)]\,dt+\check{v}^n(t,x)\,dW_t, ~~(t,x)\in Q\\
      \check{u}^n(0,x)  =& ~ 0,~~x\in\mathcal{O},\\
      \check{u}^n|_{\partial\mathcal{O}}=&~0.
      \end{array}
      \right.
    \end{split}
    \end{equation*}
Define $u^n:=\hat{u}-\check{u}^n$. Then
    \begin{align*}
      du^n(t,x)=-(-\Delta u^n(t,x)+\bar{f}(t,x)+f^n(t,x)+\nabla\cdot\hat{g}(t,x))\,dt+(v^r(t,x)-\check{v}^{n,r}(t,x))\,dW_t^r.
    \end{align*}

 Moreover, $|(u^n)^+-u^+|\leq \phi^n$ $dt\times dx\times d\mathbb{P}$ a.e.. The zero Dirichlet conditions of $\check{u}^n$ and $\hat{u}^+$ imply $(u^n)^+|_{\partial\mathcal{O}}=0$. By \cite[Lemma 3.5]{qiutang}, we have almost surely
\begin{align}\label{equation6.13}
&\int_{\mathcal{O}}\Phi(t,x,(u^n)^+(t,x))\,dx+\frac{1}{2}\int_t^T\langle \Phi''(s,(u^n)^+(s)),|(v(s)-\check{v}^n(s))1_{\{u^n>0\}}|^2\rangle \,ds\nonumber\\
=& \int_{\mathcal{O}}\Phi(T,x,(u^n)^+(T,x))\,dx-\int_t^T\int_{\mathcal{O}}\partial_s\Phi(s,x,(u^n)^+(s,x))\,dxds\nonumber\\
&+\int_t^T\langle \Phi'(s,(u^n)^+(s)),\bar{f}(s)1_{\{u^n>0\}}\rangle\,ds+\int_t^T\langle\Phi'(s,(u^n)^+(s)),f^n(s)1_{\{u^n>0\}}\rangle _{1,-1}\,ds\nonumber\\
&-\int_t^T\langle \Phi''(s,(u^n)^+(s))\partial_j (u^n)^+(s)+\partial_j \Phi'(s,(u^n)^+(s)),-\partial_{j}(u^n)^+(s)+\hat{g}^j(s)1_{\{u^n>0\}}\rangle \,ds\nonumber\\
&-\int_t^T\langle\Phi'(s,(u^n)^+(s)),(v^r(s)-\check{v}^{n,r}(s))1_{\{u^n>0\}}\rangle \,dW^r_s,\quad \forall t\in[0,T].
\end{align}
By \cite[Corollary 3.5]{qiuwei}, there exists $\bar{u}\in\mathcal{L}^2(\mathcal{P})$ such that
    \begin{align}\label{U-controlled-by-potential}
      |\hat{u}|+\phi^1 \leq \bar{u},~~dt\times dx\times d\mathbb{P}~ a.e..
    \end{align}
By (\ref{U-controlled-by-potential}) and the properties (2) and (4) of $\Phi$, there holds $dt\times dx\times d\mathbb{P}$ a.e. that
    \begin{align}\label{boundedness-Phi-prime}
      |\Phi'(t,x,(u^n)^+(t,x))| = & ~|\Phi'(t,x,(u^n)^+(t,x))-\Phi'(t,x,0)|  \nonumber\\
       \leq & C|u^n(t,x)|\nonumber\\
       = &~ C|\hat{u}(t,x)-\check{u}^n(t,x)|\nonumber\\
       \leq &~ C|\hat{u}(t,x)|+C|\check{u}(t,x)|+C|\check{u}(t,x)-\check{u}^n(t,x)|\nonumber\\
        \leq &~ C|\hat{u}(t,x)|+C|\check{u}(t,x)|+C\phi^n(t,x)\nonumber\\
        \leq &~ C(|\check{u}(t,x)|+\bar{u}(t,x)).
    \end{align}
By property (4) of $\Phi$, there holds $dt\times dx\times d\mathbb{P}$ a.e. that
    \begin{align}\label{continuity-Phi-prime}
      |\Phi'(t,x,(u^n)^+(t,x))-\Phi'(t,x,u^+(t,x))| \leq & ~C|(u^n)^+(t,x)-u^+(t,x)| \nonumber\\
      \leq & C|\check{u}^n(t,x)-\check{u}(t,x)|\nonumber\\
      \leq & C\phi^n(t,x).
    \end{align}
(\ref{boundedness-Phi-prime}), (\ref{continuity-Phi-prime}) and \cite[Proposition 3.9(ii)]{qiuwei} yield that
    \begin{align*}
        &\lim_{n\rightarrow\infty} \int_t^T\langle\Phi'(s,(u^n)^+(s)),f^n(s)1_{\{u^n>0\}}\rangle _{1,-1}\,ds  \\
     = & ~ \int_{\mathcal{O}_t}\Phi'(s,x,u^+(s,x))\,\mu(dsdx)~a.s..
    \end{align*}
Moreover,
    \begin{align*}
     & E\sup_{0\leq t \leq T}\left|\int_t^T\langle\Phi'(s,(u^n)^+(s)),(v^r(s)-\check{v}^{n,r}(s))1_{\{u^n>0\}}\rangle \,dW^r_s-\int_t^T\langle\Phi'(s,u^+(s)),v^r(s)1_{\{u>0\}}\rangle \,dW^r_s\right|  \\
      \leq & ~CE\left(\int_0^T \left|\langle\Phi'(s,(u^n)^+(s)),(v^r(s)-\check{v}^{n,r}(s))1_{\{u^n>0\}}\rangle -\langle\Phi'(s,u^+(s)),v^r(s)1_{\{u>0\}}\rangle\right|^2\,ds\right)^{\frac{1}{2}}\\
      \leq &~ CE\left(\int_0^T \left|\langle\Phi'(s,(u^n)^+(s)),v(s)1_{\{u^n>0\}}\rangle -\langle\Phi'(s,u^+(s)),v(s)1_{\{u>0\}}\rangle\right|^2\,ds\right)^{\frac{1}{2}}\\
      &+CE\left(\int_0^T \left|\langle\Phi'(s,(u^n)^+(s))-\Phi'(s,0),\check{v}^{n}(s)1_{\{u^n>0\}}\rangle\right|^2\,ds\right)^{\frac{1}{2}}\\
      \leq & C\left(E\textrm{esssup}_{0\leq t\leq T}\|(u^n)^+(t)-u^+(t)\|^2\right)^{\frac{1}{2}}\left(E\int_0^T\|v(t)\|^2\,dt\right)^{\frac{1}{2}}\\
      &+C\left(E\textrm{esssup}_{0\leq t\leq T}\|(u^n)^+(t)\|^2\right)^{\frac{1}{2}}\left(E\int_0^T\|v(t)(1_{\{u^n>0\}}-1_{\{u>0\}})\|^2\,dt\right)^{\frac{1}{2}}\\
      &+C\left(E\textrm{esssup}_{0\leq t\leq T}\|u^n(t)\|^2\right)^{\frac{1}{2}}\left(E\int_0^T\|\check{v}^n(t)\|^2\,dt\right)^{\frac{1}{2}}\\
      \leq & C\left(E\|u^n-u\|^2_{\mathcal{K}}\right)^{\frac{1}{2}}\left(E\int_0^T\|v(t)\|^2\,dt\right)^{\frac{1}{2}}\\
      &+C\left(E\|u^n\|^2_{\mathcal{K}}\right)^{\frac{1}{2}}\left(E\int_0^T\|v(t)(1_{\{u^n>0\}}-1_{\{u>0\}})\|^2\,dt\right)^{\frac{1}{2}}\\
      &+C\left(E\textrm{esssup}_{0\leq t\leq T}\|u^n(t)\|^2\right)^{\frac{1}{2}}\left(E\int_0^T\|\check{v}^n(t)\|^2\,dt\right)^{\frac{1}{2}}\\
      \rightarrow & ~0.
    \end{align*}
By the properties of $\Phi$ and the fact that $|(u^n)^+-u^+|\leq \phi^n$ $dt\times dx\times d\mathbb{P}$ a.e., the convergence of other terms can be treated analogously. Finally by letting $n\rightarrow\infty$, we obtain almost surely that
    \begin{align*}
        &\int_{\mathcal{O}}\Phi(t,x,u^+(t,x))\,dx+\frac{1}{2}\int_t^T\langle \Phi''(s,u^+(s)),|(v(s)1_{\{u>0\}}|^2\rangle \,ds\nonumber\\
        =& \int_{\mathcal{O}}\Phi(T,x,u^+(T,x))\,dx-\int_t^T\int_{\mathcal{O}}\partial_s\Phi(s,x,u^+(s,x))\,dxds\nonumber\\
        &+\int_t^T\langle \Phi'(s,u^+(s)),\bar{f}(s)1_{\{u>0\}}\rangle\,ds+\int_t^T\int_{\mathcal{O}}\Phi'(s,x,u^+(s,x))\,\mu(dsdx)\nonumber\\
        &-\int_t^T\langle \Phi''(s,u^+(s))\partial_j u^+(s)+\partial_j \Phi'(s,u^+(s)),-\partial_{j}u^+(s)+\hat{g}^j(s)1_{\{u^n>0\}}\rangle \,ds\nonumber\\
        &-\int_t^T\langle\Phi'(s,u^+(s)),v^r(s)1_{\{u>0\}}\rangle \,dW^r_s,\quad \forall t\in[0,T].
    \end{align*}
\end{proof}
\section{Some definitions associated with stochastic regular measures}\label{appendix-measure}

In general the random measure $\mu$ in (\ref{equation1.1}) can be a local time, which is not absolutely continuous w.r.t. Lebesgue measure. Hence, the Skorokhod condition $\int_Q(u-\xi)\,\mu(dt,dx)=0$ might not make sense. To give a precise meaning to the Skorohod condition, the theory of parabolic potential and capacity introduced by \cite{pierre,pierre1980} was generalized by \cite{qiuwei} to a backward stochastic framework. This subsection recalls the notion of quasi continuity  and stochastic regular measure, which are repeatedly used in the main text and in the proof of Theorem \ref{Ito-u+}. Moreover, spaces used in the proof of Theorem \ref{Ito-u+} are also presented.

First some spaces are introduced. Denote by $H_0^1(\mathcal{O})$ the first order Sobolev space vanishing on the boundary $\partial\mathcal{O}$ equipped with the norm $\|\upsilon\|^2_1:=\|\upsilon\|^2+\|\nabla\upsilon\|^2$ and by $H^{-1}(\mathcal{O})$ the dual space of $H_0^1(\mathcal{O})$. The dual pair between $H_0^1(\mathcal{O})$ and $H^{-1}(\mathcal{O})$ is denoted by $\langle\cdot,\cdot\rangle_{1,-1}$. Define $(H^{-1})^+(\mathcal{O})=\{v\in H^{-1}(\mathcal{O}):\langle \varphi,v\rangle_{1,-1}\geq 0,~\textrm{for~each}~\varphi\in H^1_0(\mathcal{O})~\textrm{and~}\varphi\geq 0\}$.

For a Hilbert space $V$, denote by $\mathcal{L}^2([0,T];V)$ the set of all $L^2([0,T];V)$ valued $(\mathcal{F}_t)$ adapted process $u$ with the norm defined as $\|u\|_{\mathcal{L}^2([0,T];V)}:=\left(E\|u\|_{L^2([0,T];V)}^2\right)^{\frac{1}{2}}<\infty$. Denote by $\mathcal{L}^2(\mathcal{O})$ the set of all $L^2(\mathcal{O})$ valued $(\mathcal{F}_t)$ adapted process $u$ with the norm $\|u\|_{\mathcal{L}^2(\mathcal{O})}:=\left(E\|u\|^2\right)^{\frac{1}{2}}<\infty$

Denote  $\mathcal{K}:=L^{\infty}([0,T];L^2(\mathcal{O}))\cap L^2([0,T];H_0^1(\mathcal{O}))$, equipped with the norm
$$\|\upsilon\|_{\mathcal{K}}:=\left(\|\upsilon\|^2_{L^{\infty}([0,T];L^2(\mathcal{O}))}+\|\upsilon\|^2_{L^2([0,T];H_0^1(\mathcal{O}))}\right)^{\frac{1}{2}}.$$
Set $\mathcal{W}=\{\upsilon\in L^2(0,T;H_0^1):\partial_t\upsilon \in L^2(0,T;H^{-1})\}$ endowed with the norm
$$\|\upsilon\|_{\mathcal{W}}=\left(\|\upsilon\|^2_{L^2(0,T;H_0^1)}+\|\partial_t\upsilon\|^2_{L^2(0,T;H^{-1})}\right)^{\frac{1}{2}},$$
where $H^{-1}$ is the dual space of $H_0^1$. Furthermore, we set
$$\mathcal{W}_T=\{\upsilon\in\mathcal{W}:\upsilon(T)=0\},~\mathcal{W}^+=\{\upsilon\in\mathcal{W}:\upsilon\geq 0\},~\mathcal{W}_T^+=\mathcal{W}_T\cap\mathcal{W}^+.$$

\begin{definition} 
We denote by $\mathcal{P}$ the set of parabolic potentials, which is the class of $\upsilon\in\mathcal{K}$ such that
$$\int_0^T-\langle\partial_t\varphi(t),\upsilon(t)\rangle \,dt+\int_0^T\langle\partial_i\varphi(t),\partial_i\upsilon(t)\rangle \,dt\geq 0,~\forall\varphi\in\mathcal{W}^+_T.$$
\end{definition}

Denote by $\mathcal{C}(Q)$ the class of continuously differentiable functions in $Q$ with compact support.
By the Hahn-Banach theorem and because $\mathcal{C}(Q)\cap\mathcal{W}_T$ is dense in $\mathcal{C}(Q)$, parabolic potentials can be represented by associated Radon measures. This leads to the following proposition, due to Pierre \cite{pierre1980}.

\begin{proposition}\label{radon-measure}
Let $\upsilon\in\mathcal{P}$. Then there exists a unique Radon measure on $[0,T)\times\mathcal{O}$, denoted by $\mu ^{\upsilon}$, such that
\begin{align*}
\forall~\varphi\in\mathcal{W}_T\cap\mathcal{C}(Q),~\int_0^T-\langle \partial_t\varphi(t),\upsilon(t)\rangle+\int_0^T\langle\partial_i\varphi(t),\partial_i\upsilon(t)\rangle \,dt=\int_0^T\int_{\mathcal{O}}\varphi(t,x)\mu^{\upsilon}(dt,dx)
\end{align*}
\end{proposition}

\begin{definition}
For any open set $A\subset[0,T)\times\mathcal{O}$, the parabolic capacity of $A$ is defined as
$$\textrm{cap}(A)=\inf\{\|\varphi\|^2_{\mathcal{W}}:\varphi\in\mathcal{W}^+,~ \varphi\geq 1 ~a.e. \textrm{~on}~ A\}.$$
For any Borel set $B\subset[0,T)\times\mathcal{O}$, its parabolic capacity is defined as
$$\textrm{cap}(B)=\inf\{\textrm{cap}(A):A\supset B, ~ A~\textrm{is~open}\}.$$
\end{definition}

\begin{definition}
A real valued function $\phi$ on $[0,T)\times\mathcal{O}$ is said to be quasi-continuous, if there exists a sequence of non-increasing open sets $A_n\subset[0,T)\times\mathcal{O}$ such that
\begin{itemize}
	\item[(1)] $\phi$ is continuous on the complement of each $A_n$;
	\item[(2)] $\lim\limits_{n\rightarrow \infty}\textrm{cap}(A_n)=0$.
\end{itemize}
\end{definition}

Denote by $\mathcal{P}_0$ the class of $\upsilon\in\mathcal{P}$ such that $\upsilon$ is quasi-continuous and $\upsilon(0)=0$ in $L^2$. Each element $\upsilon\in\mathcal{P}_0$ is called a regular potential and the associated Radon measure in Definition \ref{radon-measure} is called a regular measure.
Furthermore, let $\mathcal{L}^0(\mathcal{K})$ be the class of the measurable maps from $(\Omega, \mathcal{F}_T)$ to $\mathcal{K}$, such that each element $\upsilon\in\mathcal{L}^0(\mathcal{K})$ is an $L^2$ valued adapted process. $\mathcal{L}^0(\mathcal{P})$ and $\mathcal{L}^0(\mathcal{P}_0)$ are similarly defined as $\mathcal{L}^0(\mathcal{K})$. Moreover, set
$$\mathcal{L}^2(\mathcal{K}):=L^2(\Omega,\mathcal{F}_T;\mathcal{K})\cap \mathcal{L}^0(\mathcal{K})$$
endowed with the norm $$\|\upsilon\|_{\mathcal{L}^2(\mathcal{K})}=\left( E\|\upsilon\|^2_{\mathcal{K}}\right)^{1/2}.$$

The stochastic parabolic potential is defined as
$$\mathcal{L}^2(\mathcal{P}):=\mathcal{L}^2(\mathcal{K})\cap\mathcal{L}^0(\mathcal{P}),$$
endowed with the norm $$\|u\|_{\mathcal{L}^2(\mathcal{P})}=\|u\|_{\mathcal{L}^2(\mathcal{K})}.$$
In addition, we define the stochastic regular parabolic potential as
$$\mathcal{L}^2(\mathcal{P}_0):=\mathcal{L}^2(\mathcal{P})\cap\mathcal{L}^0(\mathcal{P}_0),$$
and the associated random Radon measure is called a stochastic regular measure.
\end{appendix}

\end{document}